\documentclass[12pt]{amsart}

\usepackage[utf8]{inputenc}
\usepackage[T1]{fontenc}
\usepackage[english]{babel}
\usepackage{lmodern}
\usepackage{xparse}

\usepackage{lipsum}
\usepackage[shortlabels]{enumitem}
\usepackage{verbatim}
\usepackage[normalem]{ulem}
\usepackage[dvipsnames]{xcolor}
\usepackage{csquotes}
\usepackage{pdfpages}
\usepackage{setspace}
\usepackage{geometry}

\usepackage{graphicx}
\usepackage{caption}
\usepackage{subcaption}
\usepackage{placeins}

\usepackage{mathtools}
\usepackage{amsmath}
\usepackage{amsthm, thmtools}
\usepackage{framed}
\usepackage{nameref}
\usepackage[colorlinks,menucolor=blue,linkcolor=blue, citecolor=blue, urlcolor=blue]{hyperref} 
\usepackage[capitalise,noabbrev]{cleveref}
\crefname{subsection}{Subsection}{Subsections}

\usepackage{amssymb}
\usepackage[mathscr]{euscript}
\usepackage{esint}
\usepackage{esvect}
\usepackage{relsize}
\usepackage{mathrsfs}
\usepackage{tcolorbox}
\usepackage{extpfeil}
\usepackage{tikz}
\usetikzlibrary{cd}
\usetikzlibrary{calc}
\usetikzlibrary{decorations.pathmorphing, decorations.pathreplacing}

\numberwithin{equation}{section}
\newtheorem{theorem}{Theorem} [section]	
\newtheorem{proposition}[theorem]	 {Proposition}	
\newtheorem{corollary}	[theorem]	 {Corollary}	
\newtheorem{lemma}	    [theorem]	 {Lemma}		
\newtheorem{definition}	[theorem]   {Definition}
\theoremstyle{definition}
\newtheorem{example}	[theorem]   {Example}

\theoremstyle{remark}

\newtheorem{notation}[theorem]{Notation}
\newtheorem{remark}[theorem]{Remark}

\newcommand{\R}{\mathbb{R}}
\newcommand{\RR}{\mathbb{R}}
\newcommand{\set}[1]{\left\{ #1 \right\}}
\newcommand{\Z}{\mathbb{Z}}
\newcommand{\ZZ}{\mathbb{Z}}
\newcommand{\N}{\mathbb{N}}

\newcommand{\te}{\tilde e}

\newcommand{\Div}{\operatorname{Div}}
\renewcommand{\div}{\operatorname{div}}
\DeclareMathOperator{\Pic}{Pic}
\DeclareMathOperator{\Prin}{Prin}
\DeclareMathOperator{\Jac}{Jac}
\DeclareMathOperator{\Prym}{Prym}
\DeclareMathOperator{\Nm}{Nm}       
\DeclareMathOperator{\Coker}{Coker}
\DeclareMathOperator{\Ker}{Ker}

\newcommand{\tf}{\mathrm{tf}}

\newcommand{\trop}{\mathrm{trop}}
\newcommand{\ud}{\mathrm{ud}}               

\renewcommand{\tilde}{\widetilde}

\newcommand{\Gam}{\widetilde{\Gamma}}
\newcommand{\romanNumeralCaps}[1]{\text{\MakeUppercase{\romannumeral #1}}}
\newcommand{\Ptilde}{\widetilde{P}}
\newcommand{\Gtilde}{\widetilde{G}}
\renewcommand{\mid}{\hspace{1pt} \middle| \hspace{1pt}}
\newcommand{\baseLL}{0}

\newcommand{\vertex}[2][1]{\filldraw[black, thick] (#2) circle (#1 * 1pt); }

\newcommand{\edge}[3][1]{\draw (#2) -- (#3); \vertex[#1]{#2} \vertex[#1]{#3} }

\newcommand{\Felix}[1]{\textcolor{magenta}{[\textbf{Felix}: #1]}}
\newcommand{\Thomas}[1]{\textcolor{red}{\textbf{Thomas}: #1}}

\newcommand*{\colorboxed}{}
\def\colorboxed#1#{%
  \colorboxedAux{#1}%
}
\newcommand*{\colorboxedAux}[3]{%
  \begingroup
    \colorlet{cb@saved}{.}%
    \color#1{#2}%
    \boxed{%
      \color{cb@saved}%
      #3%
    }%
  \endgroup
}

\usepackage{ifthen}

\newcommand{\towerI}[1]{
    \begin{tikzpicture}[baseline=\baseLL]

        \ifthenelse{#1<4}{
        \filldraw[black, thick] (-2, 1.5) circle (1pt);
        \filldraw[black, thick] (-3, 1.5) circle (1pt);}{
        \filldraw[black, thick] (-2.5, 1.5) circle (2pt) (-2.3, 1.5) node{2};
        }

        \ifthenelse{#1<3}{
        \filldraw[black, thick] (-2, 1) circle (1pt);
        \filldraw[black, thick] (-3, 1) circle (1pt);}{
        \filldraw[black, thick] (-2.5, 1) circle (2pt) (-2.3, 1) node{2};
        }
        
        \ifthenelse{#1<2}{
        \filldraw[black, thick] (-2, 0.5) circle (1pt);
        \filldraw[black, thick] (-3, 0.5) circle (1pt);}{
        \filldraw[black, thick] (-2.5, 0.5) circle (2pt) (-2.3, 0.5) node{2};
        }

        \ifthenelse{#1<1}{
        \filldraw[black, thick] (-2, 0) circle (1pt);
        \filldraw[black, thick] (-3, 0) circle (1pt);}{
        \filldraw[black, thick] (-2.5, 0) circle (2pt) (-2.3, 0) node{2};
        }
        \filldraw[black, thick] (0, 1.5) circle (1pt);
        \filldraw[black, thick] (0, 0) circle (1pt);
        \filldraw[black, thick] (0, 0.5) circle (1pt);
        \filldraw[black, thick] (0, 1) circle (1pt);
        
        \filldraw[black, thick] (2, 0.75) circle (1pt);
        \draw[->, black] (-1.5, 0.75) -- (-0.5, 0.75);
        \draw[->, black] (0.5, 0.75) -- (1.5, 0.75);
        \draw (-4, 0.75) node{$\text{\romanNumeralCaps{3}}$};
    \end{tikzpicture}
}

\newcommand{\towerII}[2]{
    \begin{tikzpicture}[baseline=\baseLL]
        \ifthenelse{#2<2}{
        \filldraw[black, thick] (-2, 1) circle (1pt);
        \filldraw[black, thick] (-3, 1) circle (1pt);}{
        \filldraw[black, thick] (-2.5, 1) circle (2pt) (-2.3, 1) node{2};
        }
        
        \ifthenelse{#2<1}{
        \filldraw[black, thick] (-2, 0.5) circle (1pt);
        \filldraw[black, thick] (-3, 0.5) circle (1pt);
        }{
        \filldraw[black, thick] (-2.5, 0.5) circle (2pt) (-2.3, 0.5) node{2};
        }
        
        \ifthenelse{#1<1}{
        \filldraw[black, thick] (-2, 0) circle (2pt) (-1.8, 0) node{2};
        \filldraw[black, thick] (-3, 0) circle (2pt) (-2.8, 0) node{2};
        }{
        \filldraw[black, thick] (-2.5, 0) circle (4pt) (-2.2, 0) node{4};
        }

        \filldraw[black, thick] (0, 0) circle (2pt) (-0.2, 0) node{2};
        \filldraw[black, thick] (0, 0.5) circle (1pt);
        \filldraw[black, thick] (0, 1) circle (1pt);
        \filldraw[black, thick] (2, 0.5) circle (1pt);
        \draw[->, black] (-1.5, 0.5) -- (-0.5, 0.5);
        \draw[->, black] (0.5, 0.5) -- (1.5, 0.5);
        \draw (-4, 0.5) node{$\text{\romanNumeralCaps{2}}$};
    \end{tikzpicture}
}

\newcommand{\towerIII}[1]{
    \begin{tikzpicture}[baseline=\baseLL]
                
        \ifthenelse{#1<2}{
        \filldraw[black, thick] (-2, 0.5) circle (2pt) (-1.8, 0.5) node{2};
        \filldraw[black, thick] (-3, 0.5) circle (2pt) (-2.8, 0.5) node{2};
        }{
        \filldraw[black, thick] (-2.5, 0.5) circle (4pt) (-2.2, 0.5) node{4};
        }
        
        \ifthenelse{#1<1}{
        \filldraw[black, thick] (-2, 0) circle (2pt) (-1.8, 0) node{2};
        \filldraw[black, thick] (-3, 0) circle (2pt) (-2.8, 0) node{2};
        }{
        \filldraw[black, thick] (-2.5, 0) circle (4pt) (-2.2, 0) node{4};
        }

        \filldraw[black, thick] (0, 0) circle (2pt) (-0.2, 0) node{2};
        \filldraw[black, thick] (0, 0.5) circle (2pt) (-0.2, 0.5) node{2};
        \filldraw[black, thick] (2, 0.25) circle (1pt);
        \draw[->, black] (-1.5, 0.25) -- (-0.5, 0.25);
        \draw[->, black] (0.5, 0.25) -- (1.5, 0.25);
        \draw (-4, 0.25) node{$\text{\romanNumeralCaps{3}}_{#1}$};
    \end{tikzpicture}
}

\newcommand{\towerIV}[2]{
    \begin{tikzpicture}[baseline=\baseLL]
                
        \ifthenelse{#1<1}{
        \filldraw[black, thick] (-2, 0) circle (3pt) (-1.7, 0) node{3};
        \filldraw[black, thick] (-3, 0) circle (3pt) (-2.7, 0) node{3};
        }{
        \filldraw[black, thick] (-2.5, 0) circle (6pt) (-2.1, 0) node{6};
        }
        
        \ifthenelse{#2<1}{
        \filldraw[black, thick] (-2, 0.5) circle (1pt);
        \filldraw[black, thick] (-3, 0.5) circle (1pt);
        }{
        \filldraw[black, thick] (-2.5, 0.5) circle (2pt) (-2.2, 0.5) node{2};
        }

        \filldraw[black, thick] (0, 0) circle (3pt) (-0.3, 0) node{3};
        \filldraw[black, thick] (0, 0.5) circle (1pt);
        \filldraw[black, thick] (2, 0.25) circle (1pt);
        \draw[->, black] (-1.5, 0.25) -- (-0.5, 0.25);
        \draw[->, black] (0.5, 0.25) -- (1.5, 0.25);
        \draw (-4, 0.25) node{$\text{\romanNumeralCaps{1}}$};
    \end{tikzpicture}
}

\newcommand{\towerV}[1]{
    \begin{tikzpicture}[baseline=\baseLL]
                
        \ifthenelse{#1<1}{
        \filldraw[black, thick] (-2, 0) circle (4pt) (-1.7, 0) node{4};
        \filldraw[black, thick] (-3, 0) circle (4pt) (-2.7, 0) node{4};
        }{
        \filldraw[black, thick] (-2.5, 0) circle (8pt) (-2.1, 0) node{8};
        }

        \filldraw[black, thick] (0, 0) circle (4pt) (-0.3, 0) node{4};
        \filldraw[black, thick] (2, 0) circle (1pt);
        \draw[->, black] (-1.5, 0) -- (-0.5, 0);
        \draw[->, black] (0.5, 0) -- (1.5, 0);
        \draw (-4, 0) node{$\text{\romanNumeralCaps{5}}_{#1}$};
    \end{tikzpicture}
}

\title{Tropical Donagi theorem}

\author{Felix Röhrle}
\address {Universit\"at T\"ubingen, Fachbereich Mathematik, Auf der Morgenstelle 10, 72076 T\"ubingen, Germany}
\email{roehrle@math.uni-tuebingen.de}

\author{Thomas Saillez}
\address{Universit\'e Libre de Bruxelles, D\'epartement de Math\'ematique, Campus de la Plaine – CP 210, Boulevard du Triomphe, B-1050 Bruxelles, Belgium}
\email{thomas.saillez@ulb.be}



\date{}

\begin{document}

\begin{abstract}
    The tropical $n$-gonal construction was introduced in recent work by the first author and D.~Zakharov and structural results for $n = 2,3$ were established. In this article we explore the construction for $n = 4$ and prove a tropical analogue of Donagi's theorem which states that the tetragonal construction is a triality which preserves Prym varieties. This confirms the speculations in previous work and establishes new results on the non-injectivity of the tropical Prym\textendash Torelli morphism. Finally, we demonstrate that the tropical $n$-gonal construction is poorly behaved under edge contractions, thus preventing any immediate moduli-theoretic perspective.
\end{abstract}

\maketitle

\setcounter{tocdepth}{1}

\section{Introduction}

Tropical geometry is a branch of mathematics which replaces the role of polynomial functions in algebraic geometry with piecewise linear functions. This yields a theory of geometry which is in many ways surprisingly similar to algebraic geometry.

This paper contributes to the development of the theory of tropical Prym varieties, which are an analogue of the algebro-geometric concept of Prym varieties. These are named after the work of Prym \cite{friedrichPrym} and were reintroduced a century later by Mumford \cite{Mumford} using modern algebro geometric language. Starting with an \'etale double cover of smooth algebraic curves $\pi:\tilde C\to C$, the Prym variety $\Prym(\tilde C / C)$ is the abelian variety given by the kernel of the norm map, i.e. the kernel of the pointwise application of $\pi$ to divisor classes of $\tilde C$.

In \cite{donagiFirst} and \cite{donagiSecond}, Donagi introduced the $n$-gonal construction on pairs of maps of curves. The construction looks at the space of sections of a pair of covers $\tilde C\to C\to \mathbb{P}^1$ where the first map is an \'etale double cover and the second one is a ramified $n$-tuple cover and produces a new curve $D$ with structural map $D \to \mathbb{P}^1$ of degree $2^n$. 

Some special cases of the $n$-gonal construction have led to structural results.
The bigonal construction ($n = 2$) naturally appeared when studying the integrability of geodesic flows on $SO(n)$ in works such as \cite{haine} and \cite{Pantazis}. The Recillas' construction is inverse to the trigonal construction ($n = 3$) and appeared in the work of Recillas \cite{Recillas}, connecting Prym varieties of trigonal curves and Jacobians of tetragonal curves. The tetragonal construction ($n = 4$) was used in \cite{donagiSecond} to discover a triality of tetragonal curves having isomorphic Prym varieties, giving a better understanding of the non-injectivity of the Prym\textendash Torelli map, with the goal of solving the Schottky problem. 

On the tropical geometry side, the notions of divisor and Jacobian were introduced in the work of Mikhalkin and Zharkov \cite{MikhalkinZharkov}. In this context, a tropical curve is a metric graph and a divisor is a formal linear combination of the points of this metric graph. A suitable notion of linear equivalence of divisors is obtained with piecewise linear functions. In this language, the Jacobian of a metric graph is the group of degree 0 divisors modulo linear equivalence. The tropical Prym variety was first defined in \cite{JensenLen}. To this end consider the kernel of the pushforward of the Jacobian in a double cover of tropical curves. This kernel might have two connected components (in case the double cover was free), so one defines the tropical Prym variety as the connected component containing zero. This definition is in pure analogy to that of the algebraic Prym variety. The undilated (i.e. \enquote{free}) Prym variety was studied further in \cite{LenUlirsch, LenZakharov} and the Prym variety was shown, using dilated double covers, to be discontinuous under edge contraction in \cite{GhoshZakharov}. Finally, in \cite{MatroidalPerspective} the authors used matroids to analyse Prym varieties and provided some initial results on the non-injectivity of the tropical Prym\textendash Torelli morphism.

The first author and D.~Zakharov introduced a tropical analogue of the $n$-gonal construction in \cite{roehrleNgonal}. The definition is a combinatorial version of the algebraic construction, where the authors replaced the projective line (which has trivial algebraic Jacobian) with metric trees (which are the metric graphs with trivial tropical Jacobian).
The authors proved the tropical analogues of the main theorem for the bigonal and trigonal construction using cohomological techniques. The tetragonal case was left open as Donagi's original proof relied on cohomological considerations that are not yet available in tropical geometry (see also \cite[Section~12.8]{BirkenhakeLange} for a more recent treatment of the theorem). We still believe that the language of tropical (co)homology introduced by \cite{IKMZ} and then further developed in \cite{GrossShokrieh_homology} is suitable for reproducing the original proof in tropical language, however, this would require establishing the tropical analogues of various cohomological formulae, which lies beyond the scope of a single paper. 

In the present article, we follow the much more combinatorial proof of Donagi's theorem developed by Izadi and Lange in \cite{izadiLange}. It turns out that this proof can be adapted to the tropical setting to establish the tropical Donagi theorem. Izadi and Lange also extended Donagi's theorem to situations where the base curve is not $\mathbb{P}^1$ and we show that the tropicalisation of the proof falls flat when the base tropical curve is not a tree. 

Roughly speaking, the tropical tetragonal construction associates to a \enquote{tower} $\tilde \Gamma \to \Gamma \to K$ consisting of a double cover $\tilde \Gamma \to \Gamma$ followed by a harmonic map $\Gamma \to K$ of degree 4 a harmonic map $\tilde \Delta \to K$ of degree 16. If the tower is \emph{orientable} (\cref{def:orientable}), this data decomposes as a disjoint union of two towers $\tilde \Lambda_i \to \Lambda_i \to K$ for $i = 1,2$ with harmonic maps of degree 2 and 4 as in the initial data. In this setting the main theorem of our paper is the following:

\begin{theorem}[tropical Donagi theorem] \label{thm:main}
    Let $\tilde \Gamma \to \Gamma \to K$ be an orientable tower consisting of a free double cover $\tilde \Gamma \to \Gamma$ and a harmonic map $\Gamma \to K$ of degree 4 with dilation profile nowhere $(2, 2)$ or $(4)$ (see \cref{def:good_types}). Then the tetragonal construction produces two more towers $\tilde \Lambda_i \to \Lambda_i \to K$ for $i = 1,2$ of the same type. In this situation we have the following.
    \begin{enumerate}
    	\item The construction is a triality, meaning that tetragonal construction applied to $\tilde \Lambda_i \to \Lambda_i \to K$ gives $\tilde \Gamma \to \Gamma \to K$ and $\tilde\Lambda_{3-i} \to \Lambda_{3-i} \to K$ as its output.
    	
    	\item If additionally the curves $\tilde \Gamma$, $\tilde\Lambda_1$, and $\tilde \Lambda_2$ are all connected and $K$ is a tree, then there are isomorphisms of principally polarised tropical abelian varieties
    	\[ \Prym(\tilde \Gamma / \Gamma) \cong \Prym(\tilde \Lambda_1 / \Lambda_1) \cong \Prym(\tilde \Lambda_2 / \Lambda_2). \]
    \end{enumerate}
\end{theorem}

A large class of orientable towers is given by \cref{lem:orientable_quotient} and \cref{prop:Donagi_connected}: 
if $K$ is a tree, or more generally, when $\tilde\Gamma$ is a fibrewise quotient of the trivial octuple cover of $K$, then the tower $\tilde\Gamma \to \Gamma \to K$ is orientable. In those cases, the connectedness of $\Gam$ implies the connectedness of both $\tilde \Lambda_1 $ and $\tilde \Lambda_2 $.

\cref{thm:main} is similar to the tropical Recillas' theorem in \cite{roehrleNgonal} in the sense that both tropical theorems substitute \enquote{\'etale} and \enquote{ramification} in the algebraic theorem by \enquote{free} and \enquote{dilation} in the tropical version. This means that the tropical theorems are really analogues of the algebraic theorems and not tropicalisations: indeed, the tropicalisation of an \'etale double cover will in general have dilation.

Finally, in \cref{sec:non-cont} we study the interaction of the tetragonal construction and the Prym\textendash Torelli morphism. It turns out that the tropical $n$-gonal construction is not continuous under edge contractions and that the tetragonal construction cannot even be modified to extend \cref{thm:main} by continuity. This shows poor behaviour of the construction in families and that the assumptions on the dilation profiles in \cref{thm:main} cannot easily be lifted. 

    
    





\subsection*{Acknowledgements} The authors would like to thank Dmitry Zakharov for suggesting the proof strategy following \cite{izadiLange} as well Rémi Delaby for help with computations in Magma. The authors would also like to thank the anonymous referee for detailed feedback.
F.~R. acknowledges funding from Deutsche Forschungsgemeinschaft (DFG, German Research Foundation) SFB\textendash TRR 195 \enquote{Symbolic Tools in Mathematics and their Application}.
T.~S. acknowledges funding from the MIS grant from the Fond national de la recherche scientifique (FNRS, Belgian Research Foundation) and the ARC grant from the Université Libre de Bruxelles (ULB, university), both grants from Špela Špenko.

\section{Preliminaries}
\label{sec-preliminaries}

\subsection{Graphs, harmonic morphisms, and tropical curves}

We consider a \textit{graph} $G$ to be a set of vertices $V(G)$, a set of half-edges $H(G)$, a root map $r:H(G)\to V(G)$ and a fixed-point-free involution $H(G)\to H(G), ~h\mapsto \Bar{h}.$
A \textit{point} of $G$ is an element of $H(G)\cup V(G)$, and by abuse of notation we will usually denote the set of points by $G$ as well. 
An \textit{edge} is a set $\set{e,\Bar{e}}$ and the set of edges is denoted $E(G)$.
We assume our set of points to always be \textbf{finite}. For a vertex $v\in V(G)$ we define the \textit{tangent space} $T_v(G)=r^{-1}(v)$ and the valence of a vertex is the cardinality of its tangent space.
For a connected graph $G$, we denote its \emph{genus} by $g(G) = |E(G)| - |V(G)| + 1$.

A \textit{morphism} of graphs from $\Gtilde$ to $G$ is a function $f:\Gtilde\to G$ such that vertices are mapped to vertices and half-edges to half-edges and which commutes with both the root map and the involution. In particular, no edge is contracted.

A \textit{harmonic morphism} from $\Gtilde$ to $G$ is a pair consisting of a morphism $f:\Gtilde\to G$ and a local degree function $d_f:\Gtilde\to \N^*$ which is constant on every edge and which satisfies that for every vertex $v\in V(\Gtilde)$ and half-edge $h\in T_{f(v)}(G)$ we have $$d_f(v)=\sum_{h'\in T_v\Gtilde\cap f^{-1}(h)}d_f(h') .$$

If $G$ is connected, it follows that the quantity $\sum_{y\in f^{-1}(x)}d_f(y)$ does not depend on the choice of $x\in G$. In that case, we call this quantity the \textit{degree} of $f$ and we denote it $\text{deg}(f)$.
A harmonic morphism $f:\Gtilde\to G$ such that the corresponding function $d_f$ is the constant 1 is called \textit{free}.

A \emph{tropical curve} is a triple $\Gamma = (G, \; \ell : E(G) \to \R_{>0}, \; g : V(G) \to \N)$ consisting of a graph $G$, an \emph{edge-length function} $\ell$, and a \emph{vertex genus function} $g$. 
The underlying \emph{metric graph} of a tropical curve $\Gamma$ is the topological space 
$$\left(\coprod_{e\in E(G)} \big[0,\ell(e) \big] \sqcup \coprod_{v\in V(G)} \set{v} \right)/\simeq$$ 
where $\simeq$ identifies start point (resp. end point) of the interval $[0, \ell(e)]$ with a vertex if the corresponding vertex is the start point (resp. end point) of the edge $e$. The metric graph is indeed a metric space (with shortest path metric, following \cite[Definition~3.1.12]{buragoMetric}), but we will not use the implied distance function. We will generally abuse notation and refer to the metric space of $\Gamma$ by $\Gamma$ as well.
A representation of the metric space $\Gamma$ by a tuple $(G, \ell, g)$ is called \emph{graph model} and throughout we will always implicitly think of tropical curves as metric spaces together with the choice of a graph model. 
Given a tropical curve $(G, \ell, g)$, its \emph{combinatorial type} is the graph $G$ together with the vertex genus function $g$, i.e. the combinatorial type contains only the discrete information of the tropical curve. The \emph{genus} of a tropical curve $\Gamma$ is the genus of the underlying graph plus the total vertex genus $g(\Gamma) = g(G) + \sum_{v \in V(G)} g(v)$. 
The vertex genus function is primarily important for the construction of moduli spaces, which we will discuss in \cref{sec:non-cont}. For the tropical Donagi theorem however, it plays no role and we will assume from here on $g \equiv 0$ and drop $g$ from the notation for tropical curves unless stated otherwise. In particular, we assume from here on the genus of a tropical curve to be equal to the genus of the underlying graph.

A \emph{(harmonic) morphism of tropical curves} $\phi : \Gamma \to \Delta$ is a harmonic morphism $f$ between the underlying graphs such that for any edge $e \in E(G)$ the lengths of $e$ and $f(e)$ are related by $\ell(e) \cdot d_f(e)=\ell(f(e))$. We note that for a harmonic morphism of graphs $f$ with an edge-length function on the target, there is a unique edge-length function on the source making $f$ a harmonic morphism of tropical curves.

A \emph{double cover} is a harmonic morphism of degree 2 and it is equivalently described with an involution $\iota : \tilde\Gamma \to \tilde\Gamma$ on the source curve. We will move freely between these two descriptions. 
The edges/vertices in $\Gamma$ with a single preimage are called \emph{dilated} and those with two preimages are called \emph{undilated} or \emph{free}.

\subsection{Divisors and linear equivalence}

We now recall basic definitions of divisors. More details can be found in \cite{lenIntroductionToDivisors}.

Let $\Gamma$ be a connected metric graph. We denote $\Div(\Gamma):=\bigoplus_{p\in\Gamma}\Z$ the group of divisors of $\Gamma$. A divisor $D$ is said to be \textit{effective} if for all $p\in\Gamma$ we have $D(p)\geq 0$. 

Let $f:\Gamma\to \R$ be a continuous piecewise linear function and let $p \in \Gamma$. A small enough neighbourhood of $p$ decomposes as a finite union of disjoint open intervals and $\set{p}$. We fix an orientation on the intervals such that they are pointing away from $p$, and this fixes the sign of the slope $f'$ on those intervals. We define $$\div(f)(p)=\sum_{h\in r^{-1}(p)}f'(h),$$
where $r^{-1}(p)$ is taken as the tangent space of $p$ for a model of $\Gamma$ where $f'$ is constant on edges.
In a more concrete way, we associate to a point $p$ the sum of the outgoing slopes of $f$ at $p$. 

A piecewise linear function $f : \Gamma \to \RR$ is called \emph{rational} if all slopes are integers. A divisor $D$ is said to be \textit{principal} if there exists a rational function $f$ such that $D=\div(f)$. Principal divisors form a subgroup $\text{Prin}(\Gamma)\subseteq\Div(\Gamma)$. 
The $\emph{degree}$ of a divisor $D\in\Div(\Gamma)$ is $\deg(D)=\sum_{p\in\Gamma} D(p)$. This sum is well-defined because $D$ has finite support. A classical result is that all divisors arising from functions have degree $0$. Denoting the set of degree $0$ divisors by $\Div^0(\Gamma)$, we define the \emph{Picard group} $\Pic^0(\Gamma) = \Div^0(\Gamma) / \Prin(\Gamma)$.

Moreover, any degree 0 divisor can be written as the divisor of a (not necessarily rational) function and this function is unique up to the addition of a constant function. For any degree 0 divisor $D$ we define the \textit{Euclidean norm} 
\begin{equation} \label{eq:pol_Pic}
	||D||^2 = \int_\Gamma (f')^2
\end{equation}

where $f$ satisfies $\div(f)=D$ and the integral is taken with respect to the Lebesgue measure. This norm allows to define a scalar product on $\Div^0(\Gamma)$ which we call the \textit{polarisation}. 

If $f:\Gamma\to\Gamma'$ is a continuous map of metric graphs and $D\in\Div(\Gamma)$, the \textit{pushforward} of $D$ is the divisor $f_\ast D\in\Div(\Gamma')$ where for $p\in\Gamma'$ we have $(f_\ast D)(p)=\sum_{q\in f^{-1}(p)}D(q)$.

\subsection{Tropical Abelian varieties}

Tropical abelian varieties can be introduced in different ways. In \cite{roehrleNgonal} the following language was developed with a view towards non-principally polarised varieties such as Prym varieties of dilated double covers. We recall these notions quickly and then explain how this can be reconciled with the divisorial perspective for free double covers (and hence naturally principally polarised Prym variety).

A \emph{tropical abelian variety}, abbreviated \emph{tav}, is a triple $(\Lambda, \Lambda', [\cdot, \cdot])$, where $\Lambda$ and $\Lambda'$ are lattices of the same rank and $[\cdot, \cdot] : \Lambda \times \Lambda' \to \RR$ is a non-degenerate pairing.
The \emph{dimension} of the tropical abelian variety is the rank of the lattices $\Lambda$ and $\Lambda'$.
A \emph{polarisation} is an injective group homomorphism $\xi : \Lambda' \to \Lambda$ such that $[ \xi(\cdot), \cdot]$ is a symmetric positive definite bilinear form.
A polarisation is called \emph{principal} if it is an isomorphism.
A tav together with a polarisation is called \emph{polarised} tropical abelian variety. 

A \emph{morphism} between tavs $(\Lambda_i, \Lambda'_i, [\cdot, \cdot]_i)$ for $i = 1,2$ consists of two $\ZZ$-linear maps $f^\# : \Lambda_2 \to \Lambda_1$ and $f_\# : \Lambda'_1 \to \Lambda'_2$ such that 
\[ [f^\#(\cdot), \cdot]_1 = [\cdot, f_\#(\cdot)]_2 \]
on all of $\Lambda_2 \times \Lambda'_1$. The pair $(f^\#, f_\#)$ is an \emph{isomorphism} if and only if both $f^\#$ and $f_\#$ are isomorphisms. If the tavs $(\Lambda_i, \Lambda'_i, [\cdot, \cdot]_i)$ are additionally equipped with polarisations $\xi_i : \Lambda'_i \to \Lambda_i$, then $(f^\#, f_\#)$ is an \emph{isomorphism} of polarised tavs if it is an isomorphism of tavs which respects the polarisations in the sense that $\xi_1 = f^\# \circ \xi_2 \circ f_\#$.

\begin{remark} \label{rem:pptav}
    We emphasise that up to isomorphism, a principally polarised tav is entirely determined by the positive definite bilinear form $[\xi(\cdot), \cdot] : \Lambda' \times \Lambda' \to \RR$, which in turn is again equivalent to giving the quadratic form $\gamma \mapsto [\xi(\gamma), \gamma]$. This statement does not hold for non-principally polarised tavs.
\end{remark}

\subsection{Jacobians}
Let $\Gamma$ be a tropical curve. Choose an orientation for the edges of the underlying graph. Then there are two canonically associated lattices to $\Gamma$, namely harmonic 1-forms with integer coefficients
\[ \Omega^1(\Gamma) = \Big\{ \sum_e a_e de \mathrel{\Big|} \text{for every vertex } v: \sum_{e \text{ entering } v} a_e - \sum_{e \text{ leaving } v} a_e = 0 \Big\}  \]
and the first homology group
\[ H_1(\Gamma, \ZZ) = \bigg\{ \sum_e b_e e \mathrel{\bigg|} 
\begin{minipage}{6.2cm}
    $\partial(\sum b_e e) = 0$, where $\partial$ is generated by $\partial(e) = \text{ end vertex } - \text{ start vertex of } e$    
\end{minipage}  \bigg\}. \]
These two lattices define the \emph{(tropical) Jacobian} of $\Gamma$ as follows 
\[ \Jac(\Gamma) = \big( \Omega^1(\Gamma), H_1(\Gamma, \ZZ), [\cdot, \cdot] \big), \]
where 
$[\cdot, \cdot] : \Omega^1(\Gamma) \times H_1(\Gamma, \ZZ) \longrightarrow \RR$ is the integration pairing generated by
\[ [ de, e'] = \int_{e'} de = \begin{cases}
    \ell(e) & \text{if $e = e'$ and the orientations match}, \\
    -\ell(e) & \text{if $e = e'$ with opposing orientation}, \\
    0 & \text{else.}
\end{cases} \]
The Jacobian is a tropical abelian variety of dimension $g(\Gamma)$ and it carries a natural principal polarisation, which is given by the isomorphism $H_1(\Gamma, \ZZ) \to \Omega^1(\Gamma)$ induced by $e \mapsto de$.
The tropical Abel\textendash Jacobi theorem states that $\Jac(\Gamma) \cong \Pic^0(\Gamma)$ as principally polarised tropical abelian varieties and the idea behind this isomorphism is as follows: 
let $D\in\Div^0(\Gamma)$ be a divisor and choose a model of $\Gamma$ such that $D$ is supported on the vertices of $\Gamma$. Then there exists a (not necessarily rational) function $f:\Gamma\to \R$ which is piecewise linear, unique up to addition of a constant and such that $\div(f)=D$. This function has constant slopes on the edges of $\Gamma$, so the data of $f:\Gamma\to\R$ modulo the constant is equivalent to the slopes of $f$ on the edges, i.e. $f$ can be summarised as $\sum_{e\in E}f'(e) \cdot e\in\R^E$. This computation turns degree 0 divisors into formal sums of edges and it preserves the polarisation in the sense that 
\[ ||D||^2 = \Big[\zeta\Big(\sum_{e\in E}f'(e) \cdot e \Big), \sum_{e\in E}f'(e) \cdot e \Big] \]
where $[\zeta(\cdot), \cdot]$ is extended to real coefficients by linearity. More details about this isomorphism are the subject of \cite[Sections~2 and~3]{BakerFaberMetric}.

\subsection{Pryms}

Let $\pi : \tilde \Gamma \to \Gamma$ be a double cover of tropical curves. It induces the \emph{norm homomorphism} 
\[ \Nm = (\pi^\ast, \pi_\ast) : \Jac(\tilde\Gamma) \longrightarrow \Jac(\Gamma). \]
We define the \emph{(divisorial) Prym variety} as the kernel of the norm homomorphism, i.e.
\[ \Prym(\tilde\Gamma / \Gamma) = (\Ker \Nm)_0 = \big( (\Coker \pi^\ast)^\tf, \Ker \pi_\ast, [\cdot, \cdot] \big),  \]
where $()^\tf$ denotes the torsion free quotient and the pairing $[\cdot, \cdot]$ is simply the restriction of the integration pairing on $\Jac(\tilde\Gamma)$.
The Prym variety is a tropical abelian variety of dimension $g(\tilde \Gamma) - g(\Gamma)$.
We endow $\Prym(\tilde\Gamma / \Gamma)$ with an induced polarisation by pulling back the principal polarisation of $\Jac(\tilde\Gamma)$, i.e.
\[ \xi : \Ker \pi_\ast \lhook\joinrel\longrightarrow H_1(\tilde \Gamma, \ZZ) \longrightarrow \Omega^1(\tilde \Gamma) \xtwoheadrightarrow[]{} (\Coker \pi^\ast)^\tf. \]
This induced polarisation on Prym is in general not principal. For free double covers $\tilde \Gamma \to \Gamma$, which is the setting of our main theorem, the induced polarisation is always twice a principal polarisation $\zeta = \xi / 2$ and thus \cref{rem:pptav} applies.
In general, \cite{roehrleNgonal} provides the construction of the \emph{continuous Prym variety}, which is a finite cover of $\Prym(\tilde\Gamma / \Gamma)$ and which naturally carries an induced polarisation which is twice a principal polarisation.

In view of our examples in \cref{sec:examples} we will now describe the bilinear form $[\zeta(\cdot), \cdot]$ for $\Prym(\tilde\Gamma / \Gamma)$ of a free double cover in more detail. 
For this we use the following convention. If $e \in E(\Gamma)$ is an undilated edge, endowed with an orientation, then we denote the two preimages of $e$ by $\te^+$ and $\te^-$ in some order and endow them with the orientation compatible with that of $e$, meaning such that $\pi_\ast(\te^\pm) = e$.
In \cite{roehrleNgonal} the first author and Zakharov presented an explicit basis for $\Ker \pi_\ast$ consisting of elements of the form 
\[ \alpha = \sum_{e \in E_\ud(\Gamma)} a_e(\tilde e^+ - \tilde e^-) \]
where $E_\ud(\Gamma)$ is the set of undilated edges of $\Gamma$. We obtain the following formula:
\begin{equation} \label{eq:pairing_Prym}
    \bigg[\zeta \Big(\sum_e a_e (\tilde e^+ - \tilde e^-) \Big), \sum_e b_e (\tilde e^+ - \tilde e^-) \bigg] = \sum_{e \in E_\ud(\Gamma)} a_e b_e \ell(e).
\end{equation}

For the proof strategy of \cref{thm:main}, however, the divisorial view on Prym is more beneficial. To make this precise, note that the (divisorial) Prym variety embeds into $\Jac(\tilde\Gamma)$ simply by restricting the pairing. 
Hence the isomorphism $\Jac(\tilde \Gamma) \cong \Pic^0(\tilde \Gamma)$ restricts to an isomorphism of principally polarised tavs between $\Prym(\tilde \Gamma / \Gamma)$ and the subvariety of \emph{antisymmetric divisors} in $\Pic^0(\tilde\Gamma)$, i.e. divisors of the form $D - \iota_\ast D$, where $D \in \Div(\tilde \Gamma)$ can be taken effective.

\section{The tropical tetragonal construction}

In \cite[Section~2]{roehrleNgonal} the authors defined the tropical $n$-gonal construction for general $n$. 
Although we will want to apply the construction to metric graphs in order to state and prove our main theorem, the construction itself is purely combinatorial and hence this section will be phrased in terms of (unmetrised) graphs.
We will recall the $n$-gonal construction in full generality before we specialise to $n = 4$. It turns out that the construction is very poorly behaved from a moduli-theoretic point of view because it is not compatible with edge contractions in general. We will explore these shortcomings in \cref{sec:non-cont} and demonstrate that no continuation of the $n$-gonal construction through edge contractions can be found.

\subsection{The tropical $n$-gonal construction}

We may define the following notions of pullback and pushforward: let $f : G \to K$ and $\pi : \tilde G \to G$ be harmonic maps of graphs. We will always assume that $K$ is connected. Using $\Gtilde$, $G$, and $K$ to represent the respective sets of points of these graphs, define

\begin{equation*}
\begin{aligned}
    f^*:\Z^K&\longrightarrow\Z^G,   
    &&\qquad\text{and}\qquad &
    \pi_*:\Z^{\widetilde{G}} &\longrightarrow\Z^G,\\
     x&\longmapsto \sum_{y\in f^{-1}(x)}d_f(y)y,    
    &&&x &\longmapsto\pi(x).
\end{aligned}
\end{equation*}
The following is the tropical Donagi construction as it was defined in \cite[Section~2.2]{roehrleNgonal}.

\begin{definition}[tropical Donagi construction] \label{def:Donagi_construction}
    Let $\pi : \tilde G \to G$ be a double cover and $f : G \to K$ a harmonic map of degree $n$. We define the set
    \[\Ptilde_{\Gtilde} = \set{D \in \Z^{\widetilde{G}} \mid D\geq 0 \text{ and } \exists x\in K \text{ such that } \pi_*(D)=f^*(x)}.\]
    Note that the $x$ in the definition of $\Ptilde_{\Gtilde}$ is necessarily unique. We say that $D \in \Ptilde_{\Gtilde}$ is a vertex (resp. half-edge) if and only if $x$ is a vertex (resp. half-edge). Define the root and involution maps of $\tilde P_{\tilde G}$ by
    \[ r\Big(\sum a_h h \Big) = \sum a_h r(h) \qquad \text{ and } \qquad \overline{\sum a_h h } = \sum a_h \overline{h}. \]
    This gives $\Ptilde_{\Gtilde}$ the structure of a graph.    
    We associate the map $\widetilde{p}:\Ptilde_{\Gtilde}\to K$ mapping $D$ to $x$ if $\pi_*(D)=f^*(x)$. 
    Furthermore, we define the local degree function 
    \[ d_{\tilde p}(D) = \prod_{y \in f^{-1}(x) \text{ free}} \binom{d_f(y)}{a_{\tilde y^+}} \prod_{y \in f^{-1}(x) \text{ dilated}} 2^{d_f(y)}. \]
\end{definition}

By \cite[Proposition~2.1]{roehrleNgonal} $\tilde p$ is a harmonic morphism of degree $2^n$. 
The Donagi construction on a tower of tropical curves is defined as the Donagi construction on the underlying graphs and the edge length function of $\tilde P_{\tilde G}$ is the only one such that the morphism $\tilde p : \Ptilde_{\Gtilde}\to K$ is harmonic. This will be explored in \cref{sec-thm}.
\medskip

For the rest of the paper we specialise to $n=4$ and recall the naming system \cite[Definition~2.8]{roehrleNgonal} for the combinatorial types which occur in the context of the tetragonal construction. We use the term \emph{tower} to refer to a pair of covers $\Gtilde\to G\to K$ where $\Gtilde\to G$ is a double cover and $G\to K$ is a harmonic map of degree 4.

\begin{definition} \label{def:good_types}
    Let $f : G \to K$ be a harmonic map of degree 4. We say that a point $x \in K$ is of type $\romanNumeralCaps{1}$, $\romanNumeralCaps{2}$, $\romanNumeralCaps{3}$, if the dilation profile of $f$ over $x$ is $(3, 1)$, $(2, 1, 1)$, $(1, 1, 1, 1)$, respectively, see \cref{fig:fibre_types}. A tower $\tilde G \to G \to K$ is called \emph{generic} if every point of $K$ is of type $\romanNumeralCaps{1}$, $\romanNumeralCaps{2}$, or $\romanNumeralCaps{3}$ with respect to $G \to K$ and the double cover $\Gtilde\to G$ is free.
\end{definition}

\begin{figure}
    \begin{center}
    	\begin{tikzpicture}
    		\draw (-2.5, 0) node[anchor = center] {$\tilde G$};
    		\draw (0, 0) node[anchor = center] {$G$};
    		\draw (2, 0) node[anchor = center] {$K$};
    		\draw[->, black] (-1.5, 0) -- (-0.5, 0);
    		\draw[->, black] (0.5, 0) -- (1.5, 0);
    		\draw[white] (-4, 0.75) node{$\text{\romanNumeralCaps{3}}$};
    	\end{tikzpicture}
    
    	\begin{tikzpicture}
    		\draw[black](-5,0) -- (5,0);
    		\filldraw[white](0,-0.2);
    	\end{tikzpicture}
        \towerIV{0}{0}
        \begin{tikzpicture}
         \draw[black](-5,0) -- (5,0);
         \filldraw[white](0,-0.2);
        \end{tikzpicture}
        \towerII{0}{0}
        \begin{tikzpicture}
         \draw[black](-5,0) -- (5,0);
         \filldraw[white](0,-0.2);
        \end{tikzpicture}
        \towerI{0}
    \end{center}
    
    \caption{Possible fibres in a generic tower. Indicated are dilation factors with respect to $K$ and dilation is additionally visualised by thickness.}
    \label{fig:fibre_types}
\end{figure}
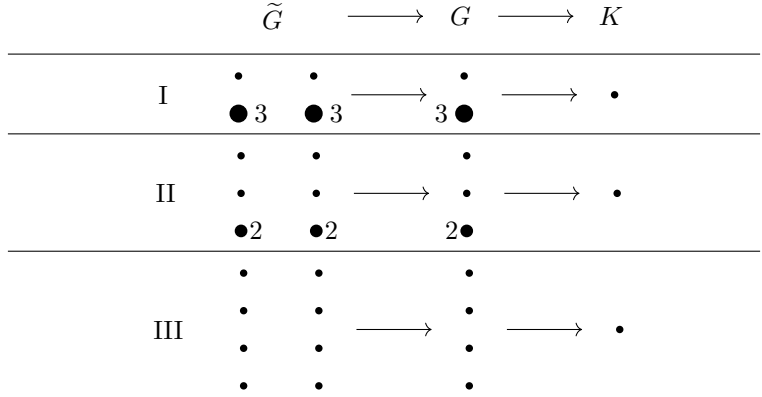

\subsection{Orientability and connectedness}

In \cite[Section 2.2]{roehrleNgonal}, the authors factorise $\tilde p : \tilde P_{\tilde G}\to K$ as 
\begin{equation}\label{eq:factorise_tetragonal}
	\Ptilde_{\Gtilde}\longrightarrow \Ptilde_{\Gtilde}/\iota\longrightarrow\tilde K\longrightarrow K,
\end{equation}
where the respective degrees of these harmonic maps are 2, 4, and 2, respectively.
We will now explain this factorisation.

First, we abuse notation and use $\iota$ to indicate the involution defining the double cover $\Gtilde \to G$ as well as the involution on $\Ptilde_{\Gtilde}$ induced by it via 
\[ \iota\Big( \sum a_h h\Big) = \sum a_h \iota(h). \]
The first map in \cref{eq:factorise_tetragonal} is then the quotient of $\tilde P_{\tilde G}$ by $\iota$. 

Next, $\tilde K \to K$ is the orientation double cover, which is constructed by a fibrewise equivalence relation. In the case of a free double cover in the original tower, the relation is defined as follows. Fix $x\in K$ and let $D = \sum_{y\in f^{-1}(x)}a_{\tilde y^+}\tilde y^+ +a_{\tilde y^-}\tilde y^-$ and  $E = \sum_{y\in f^{-1}(x)}b_{\tilde y^+}\tilde y^+ +b_{\tilde y^-}\tilde y^-$ be points in $\tilde P_{\tilde G}$. 
Now $D$ and $E$ are defined as equivalent if and only if $\sum_y a_{\tilde y^+}-b_{\tilde y^+}$ is even. Applying $\iota$ to $D$ swaps all the coefficients $a_{\tilde y^+}$ with the coefficients $a_{\tilde y^-}$. Remembering that $\sum a_{\tilde y^+}+a_{\tilde y^-}=4$ we see that $\sum a_{\tilde y^+}$ and $\sum a_{\tilde y^-}$ always have the same parity. This explains the rest of \cref{eq:factorise_tetragonal}. Note that $\tilde K \to K$ is always free.

\begin{definition} \label{def:orientable}
    A tower $\widetilde{G}\to G\to K$ is \textbf{orientable} if its orientation double cover $\tilde K \to K$ is trivial.
\end{definition}

If a tower is orientable, then the tetragonal construction decomposes as a disjoint union of two towers with degrees 2 and 4 and each of these two towers is defined over the connected components of $\tilde K = K \sqcup K$, i.e. over $K$.
More explicitly, when $\Gtilde\to G\to K$ is an orientable tower, the tetragonal construction provides $\Ptilde_{\Gtilde}\to K\coprod K\to K$, that we can split into two covers $\tilde C_i\to K$ for $i=1,2$. Since 4 is even, applying the involution $\iota$ of $\Gtilde$ on the orientation double cover is the identity, therefore $\iota$ restricts to both $\tilde C_1$ and $\tilde C_2$, allowing to produce two new towers $\tilde C_i\to C_i\to K$ for $i=1,2$. 
\medskip

Whenever $K$ is a tree, the freeness of the orientation cover $\tilde K\to K$ implies that it is trivial. In order to describe a larger class of examples of orientable towers, we need the following terminology.

\begin{definition}
	We say that $g:\Gtilde\to K$ is a \emph{fibrewise quotient} of $h:H\to K$ if both are harmonic maps of the same degree and there is a quotient map $s:H\to\Gtilde$ such that $g\circ s = h$ and for every point $p\in \Gtilde$ we respect the local degree function, i.e. $d_g(p) = \sum_{i\in s^{-1}(p)}d_h(i).$
	
	For a tower $\Gtilde\to G\to K$ to be a \emph{fibrewise quotient of the trivial octuple cover}, we want the following commutative diagram:	
	$$\begin{tikzcd}[column sep = small]
		\coprod_{i=1}^8 K\arrow[rr]\arrow[d]\arrow[drr] &  & \Gtilde\arrow[d] \\
		\coprod_{i=1}^4 K \arrow[rr] \arrow[dr]& & G\arrow[dl] \\
		& K, &
	\end{tikzcd}$$
	where $\Gtilde\to G$ is a quotient of $\coprod_{i=1}^8 K\to G$ and $G\to K$ is a quotient of $\coprod_{i=1}^4 K\to K$.	
\end{definition}

\begin{lemma}\label{lem:orientable_quotient}
    Let $\widetilde{G}\to G\to K$ be a tower such that $\Gtilde\to G$ is a free double cover and $\Gtilde\to K$ is a fibrewise quotient of the trivial octuple cover. Then this tower is orientable.
\end{lemma}

\begin{proof}
    Suppose the orientation double cover is connected. Since it is free, there must exist a cycle $S^1\subset K$ such that restricting the orientation double cover to that cycle gives the connected double cover of $S^1$.
    Then, we restrict the quotient to that $S^1$ to obtain the diagram 
    
    $$\begin{tikzcd} [column sep=small]
            \coprod_{i=1}^8 S^1\arrow[rr]\arrow[d, "2"']\arrow[drr] &  & \tilde D\arrow[d, "2"] \\
            \coprod_{i=1}^4 S^1 \arrow[rr] \arrow[dr, "4"']& & D\arrow[dl, "4"] \\
            & S^1, &
    \end{tikzcd}$$         
    where the degrees are as indicated and the first map is not harmonic.
    Applying the tetragonal construction to both $\coprod_{i=1}^{8}S^1\to \coprod_{i=1}^{4}S^1\to S^1$ and $\widetilde{D}\to D\to S^1$ gives $$\Ptilde_{\coprod_{i=1}^{8}S^1}=\coprod_{i=1}^{16}S^1\longrightarrow \Ptilde_{\tilde D} \overset{8}{\longrightarrow} S^1 \overset{2}{\longrightarrow} S^1,$$
    where the first map comes from the quotient $\coprod_{i=1}^{8}K\to\widetilde{G}$ and the last one is the orientation double cover of $\Ptilde_{\tilde D}$. Nonetheless, it is impossible to obtain a commutative triangle 
    \begin{center}
        \begin{tikzcd}
            \coprod_{i=1}^{16} S^1\arrow["8", r]\arrow["16"', d] & S^1\arrow["2",dl]\\
            S^1
        \end{tikzcd} 
    \end{center}
    of free covers because the fundamental group is represented trivially by the degree 16 cover but non-trivially by the double one. Therefore, the orientation double cover must be disconnected and hence the trivial double cover.
\end{proof}

\begin{remark}
	The tetragonal construction applied to a quotient of the trivial octuple cover will output a quotient of the trivial sexdecuple cover and thus by \cref{lem:orientable_quotient} we know that the tetragonal construction will provide two towers that are quotients of the trivial octuple cover.
\end{remark}

\begin{notation}\label{notation}
    In a generic orientable tower $\Gtilde\to G\to K$ and a fixed point $x\in K$ of fibre type $\romanNumeralCaps{3}$, we name the points in the fibre over $x$ as $x_1,x_2,x_3,x_4\in G$. The preimages of $x_i$ are named $x_i^+$ and $x_i^-$. We use the convention that $\sum_{i=1}^4x_i^+$ is a point of $\tilde C_1$ (one of the two outputs towers of the tetragonal construction). 
    For any $z =x_1^{\epsilon_1}+x_2^{\epsilon_2}+x_3^{\epsilon_3}+x_4^{\epsilon_4} \in \tilde P_{\tilde G}$, we use the short-hand notation $z=x^{\epsilon_1\epsilon_2\epsilon_3\epsilon_4}$.

    If the point $x$ is of fibre type $\romanNumeralCaps{1}$ or $\romanNumeralCaps{2}$ we give the dilated point in $f^{-1}(x)$ as many names $x_i$ as the dilation factor of the point. For example one can have a fibre of type $\romanNumeralCaps{2}$ with $x_1=x_2$ in $G$ and with that choice $x_1^+=x_2^+$, $x_1^-=x_2^-$ in $\Gtilde$. It also implies equalities in the tetragonal construction such as $x^{+-+-}=x^{-++-}$.

    Such choice of naming the fibres is called a coordinate system.
\end{notation}

\begin{example}
    The tower $S^1 \to S^1 \to S^1$  consisting of the connected free double cover of $S^1$ followed by the connected free quadruple cover of $S^1$ is not orientable. Indeed, if we label the eight preimages of a fixed point (in order) $x\in S^1$ by $x_1^+,\dots, x_4^+,x_1^-,\dots,x_4^-$, then the action of the generator of $\pi_1(S^1)$ turns the divisor $x_1^++x_2^++x_3^++x_4^+$ into $x_2^++x_3^++x_4^++x_1^-$, which maps to two different points in the orientation double cover. Therefore, the orientation double cover is the connected free double cover of $S^1$.
\end{example}

The following proposition is the main structural result on the combinatorial properties of the tetragonal construction and establishes the first part of \cref{thm:main}.

\begin{proposition} \label{prop:triality}
    The tropical tetragonal construction is a triality on orientable generic towers.  
    More explicitly, if $\tilde G \to G \to K$ is orientable and generic and $\tilde C_i\to C_i\to K$ for $i=1,2$ are the output towers of the tetragonal construction, then applying the tetragonal construction to $\tilde C_i\to C_i\to K$ outputs $\Gtilde\to G \to K$ and $\tilde C_{3-i}\to C_{3-i}\to K$.
\end{proposition}
\begin{proof}
    Let $\tilde G \to G \to K$ be a tower and let $x \in K$ be a point (vertex or edge). 

    We first assume that the type of $x$ is $\romanNumeralCaps{3}$. Then we choose a naming of the fibre $\{x_1^\pm, x_2^\pm, x_3^\pm, x_4^\pm \} \subseteq \tilde G$. We write $x^{\pm \pm\pm \pm} = x_1^\pm + x_2^\pm + x_3^\pm + x_4^\pm$ as short-hand notation. With this the even component $\tilde C_1$ of the tetragonal construction has as a fibre over $x$ (ignoring the coloured boxes for now):
    \begin{center}
        \begin{tikzcd}[row sep = 0.3cm, column sep = tiny, execute at end picture={
            \draw[draw=RoyalBlue] (B.south west) rectangle (C.north east);
            \draw[draw=LimeGreen] ($(A.south west) + (-0.1, -0.1)$) rectangle ($(C.north east) + (0.1, 0.1)$);
            \draw[draw=RoyalBlue] (E.south west) rectangle (F.north east);
            \draw[draw=LimeGreen] ($(D.south west) + (-0.1, -0.1)$) rectangle ($(F.north east) + (0.1, 0.1)$);}]
             x^{++++} & |[alias = A]| x^{++--} & |[alias = B]| x^{+-+-} & |[alias = C]| x^{+--+} \\ 
             x^{----} & |[alias = D]| x^{--++} & |[alias = E]| x^{-+-+} & |[alias = F]| x^{-++-}
        \end{tikzcd}
    \end{center}
	The involution $\iota$ on $\tilde C_1$ is given by swapping the rows.
     Taking the top row as the "positive" points, we compute the fibre over $x$ of the even component of the tetragonal construction applied to $\tilde C_1 \to C_1 \to K$ and define a map to $\tilde G$ as follows:
    \begin{center}
        \begin{tikzcd}[row sep = 0.1cm, column sep = tiny, execute at end picture={
            \draw[draw=RoyalBlue] (A.south west) rectangle (B.north east);
            \draw[draw=LimeGreen] ($(A.south west) + (-0.1, -0.1)$) rectangle ($(C.north east) + (0.1, 0.1)$);
            \draw[draw=RoyalBlue] (D.south west) rectangle (E.north east);
            \draw[draw=LimeGreen] ($(D.south west) + (-0.1, -0.1)$) rectangle ($(F.north east) + (0.1, 0.1)$);
            \draw[draw=RoyalBlue] (a.south west) rectangle (b.north east);
            \draw[draw=LimeGreen] ($(a.south west) + (-0.1, -0.1)$) rectangle ($(c.north east) + (0.1, 0.1)$);
            \draw[draw=RoyalBlue] (d.south west) rectangle (e.north east);
            \draw[draw=LimeGreen] ($(d.south west) + (-0.1, -0.1)$) rectangle ($(f.north east) + (0.1, 0.1)$);}]
                 x^{++++} + x^{++--} + x^{+-+-} + x^{+--+} & \longmapsto & x_1^+ \\ 
                 |[alias = c]| x^{++++} + x^{++--} + x^{-+-+} + x^{-++-} & \longmapsto & |[alias = C]| x_2^+ \\
                 |[alias = b]| x^{++++} + x^{--++} + x^{+-+-} + x^{-++-} & \longmapsto & |[alias = B]| x_3^+ \\
                 |[alias = a]| x^{++++} + x^{--++} + x^{-+-+} + x^{+--+} & \longmapsto & |[alias = A]| x_4^+ \\
                 x^{----} + x^{--++} + x^{-+-+} + x^{-++-} & \longmapsto & x_1^- \\ 
                 |[alias = f]| x^{----} + x^{--++} + x^{+-+-} + x^{+--+} & \longmapsto & |[alias = F]| x_2^- \\
                 |[alias = e]| x^{----} + x^{++--} + x^{-+-+} + x^{+--+} & \longmapsto & |[alias = E]| x_3^- \\
                 |[alias = d]| x^{----} + x^{++--} + x^{+-+-} + x^{-++-} & \longmapsto & |[alias = D]| x_4^-
        \end{tikzcd}
    \end{center}
    The (set-theoretic) map to the fibre of the original tower is based on the following rule: for every sum, there is always exactly one point of the original tower which is common to all the summands. We map the sum to this point.
    This establishes a canonical set-theoretic bijection between the even component of the tetragonal construction of $\tilde C_1$ with the original tower. 

    For the odd component of $\tilde P_{\tilde C_1}$ we define the following map to $\tilde C_2$: 

    \begin{center}
        \begin{tikzcd}[row sep = 0.1cm, column sep = tiny, execute at end picture={
            \draw[draw=RoyalBlue] (B.south west) rectangle (C.north east);
            \draw[draw=LimeGreen] ($(A.south west) + (-0.1, -0.1)$) rectangle ($(C.north east) + (0.1, 0.1)$);
            \draw[draw=RoyalBlue] (E.south west) rectangle (F.north east);
            \draw[draw=LimeGreen] ($(D.south west) + (-0.1, -0.1)$) rectangle ($(F.north east) + (0.1, 0.1)$);
            \draw[draw=RoyalBlue] (b.south west) rectangle (c.north east);
            \draw[draw=LimeGreen] ($(a.south west) + (-0.1, -0.1)$) rectangle ($(c.north east) + (0.1, 0.1)$);
            \draw[draw=RoyalBlue] (e.south west) rectangle (f.north east);
            \draw[draw=LimeGreen] ($(d.south west) + (-0.1, -0.1)$) rectangle ($(f.north east) + (0.1, 0.1)$);}]
         |[alias = c]| x^{++++} + x^{++--} + x^{+-+-} + x^{-++-} & \longmapsto & |[alias = C]| x^{+++-} \\ 
         |[alias = b]| x^{++++} + x^{++--} + x^{-+-+} + x^{+--+} & \longmapsto & |[alias = B]| x^{++-+} \\
         |[alias = a]| x^{++++} + x^{--++} + x^{+-+-} + x^{+--+} & \longmapsto & |[alias = A]| x^{+-++} \\
         x^{++++} + x^{--++} + x^{-+-+} + x^{-++-} & \longmapsto & x^{-+++} \\
         |[alias = f]| x^{----} + x^{--++} + x^{-+-+} + x^{+--+} & \longmapsto & |[alias = F]| x^{---+} \\ 
         |[alias = e]| x^{----} + x^{--++} + x^{+-+-} + x^{-++-} & \longmapsto & |[alias = E]| x^{--+-} \\
         |[alias = d]| x^{----} + x^{++--} + x^{-+-+} + x^{-++-} & \longmapsto & |[alias = D]| x^{-+--} \\
         x^{----} + x^{++--} + x^{+-+-} + x^{+--+} & \longmapsto & x^{+---}
        \end{tikzcd}
    \end{center}
    The set-theoretic map of fibres is obtained via the following rule: any point of $\tilde G$ appears either three or one time in those sums. We take the sum of those that appear 3 times.
    
    In a similar way one can define a bijection between the odd component of $\tilde P_{\tilde C_2}$ and the original tower: the points in $\tilde C_2$ are
    \begin{center}
        \begin{tikzcd}[row sep = 0.3cm, column sep = tiny, execute at end picture={
            \draw[draw=RoyalBlue] (A.south west) rectangle (B.north east);
            \draw[draw=LimeGreen] ($(A.south west) + (-0.1, -0.1)$) rectangle ($(C.north east) + (0.1, 0.1)$);
            \draw[draw=RoyalBlue] (D.south west) rectangle (E.north east);
            \draw[draw=LimeGreen] ($(D.south west) + (-0.1, -0.1)$) rectangle ($(F.north east) + (0.1, 0.1)$);}]
                 |[alias = A]| x^{+++-} & |[alias = B]| x^{++-+} & |[alias = C]| x^{+-++} & x^{-+++}\\ 
                 |[alias = D]| x^{---+} & |[alias = E]| x^{--+-} & |[alias = F]| x^{-+--} & x^{+---}
        \end{tikzcd}
    \end{center}
    The two components of $\tilde P_{\tilde C_2}$ are mapped to $\tilde G$ and $\tilde C_1$ by
    \begin{center}
        \begin{tikzcd}[row sep = 0.1cm, column sep = tiny, execute at end picture={
            \draw[draw=RoyalBlue] (A.south west) rectangle (B.north east);
            \draw[draw=LimeGreen] ($(A.south west) + (-0.1, -0.1)$) rectangle ($(C.north east) + (0.1, 0.1)$);
            \draw[draw=RoyalBlue] (D.south west) rectangle (E.north east);
            \draw[draw=LimeGreen] ($(D.south west) + (-0.1, -0.1)$) rectangle ($(F.north east) + (0.1, 0.1)$);
            \draw[draw=RoyalBlue] (a.south west) rectangle (b.north east);
            \draw[draw=LimeGreen] ($(a.south west) + (-0.1, -0.1)$) rectangle ($(c.north east) + (0.1, 0.1)$);
            \draw[draw=RoyalBlue] (d.south west) rectangle (e.north east);
            \draw[draw=LimeGreen] ($(d.south west) + (-0.1, -0.1)$) rectangle ($(f.north east) + (0.1, 0.1)$);}]
                 x^{+++-} + x^{++-+} + x^{+-++} + x^{+---} & \longmapsto & x_1^+ \\ 
                 |[alias = c]| x^{+++-} + x^{++-+} + x^{-+--} + x^{-+++} & \longmapsto & |[alias = C]| x_2^+ \\
                 |[alias = b]| x^{+++-} + x^{--+-} + x^{+-++} + x^{-+++} & \longmapsto & |[alias = B]| x_3^+ \\
                 |[alias = a]| x^{---+} + x^{++-+} + x^{+-++} + x^{-+++} & \longmapsto & |[alias = A]| x_4^+ \\
                 x^{---+} + x^{--+-} + x^{-+--} + x^{-+++} & \longmapsto & x_1^- \\ 
                 |[alias = f]| x^{---+} + x^{--+-} + x^{+-++} + x^{+---} & \longmapsto & |[alias = F]| x_2^- \\
                 |[alias = e]| x^{---+} + x^{++-+} + x^{-+--} + x^{+---} & \longmapsto & |[alias = E]| x_3^- \\
                 |[alias = d]| x^{+++-} + x^{--+-} + x^{-+--} + x^{+---} & \longmapsto & |[alias = D]| x_4^-
        \end{tikzcd} 
    \end{center}
    and
    \begin{center}
        \begin{tikzcd}[row sep = 0.1cm, column sep = tiny, execute at end picture={
            \draw[draw=RoyalBlue] (A.south west) rectangle (B.north east);
            \draw[draw=LimeGreen] ($(A.south west) + (-0.1, -0.1)$) rectangle ($(C.north east) + (0.1, 0.1)$);
            \draw[draw=RoyalBlue] (D.south west) rectangle (E.north east);
            \draw[draw=LimeGreen] ($(D.south west) + (-0.1, -0.1)$) rectangle ($(F.north east) + (0.1, 0.1)$);
            \draw[draw=RoyalBlue] (a.south west) rectangle (b.north east);
            \draw[draw=LimeGreen] ($(a.south west) + (-0.1, -0.1)$) rectangle ($(c.north east) + (0.1, 0.1)$);
            \draw[draw=RoyalBlue] (d.south west) rectangle (e.north east);
            \draw[draw=LimeGreen] ($(d.south west) + (-0.1, -0.1)$) rectangle ($(f.north east) + (0.1, 0.1)$);}]
                 x^{+++-} + x^{++-+} + x^{+-++} + x^{-+++} & \longmapsto & x^{++++} \\ 
                 |[alias = c]| x^{+++-} + x^{++-+} + x^{-+--} + x^{+---} & \longmapsto & |[alias = C]| x^{++--} \\
                 |[alias = b]| x^{+++-} + x^{--+-} + x^{+-++} + x^{+---} & \longmapsto & |[alias = B]| x^{+-+-} \\
                 |[alias = a]| x^{---+} + x^{++-+} + x^{+-++} + x^{+---} & \longmapsto & |[alias = A]| x^{+--+} \\
                 x^{---+} + x^{--+-} + x^{-+--} + x^{+---} & \longmapsto & x^{----} \\ 
                 |[alias = f]| x^{---+} + x^{--+-} + x^{+-++} + x^{-+++} & \longmapsto & |[alias = F]| x^{--++} \\
                 |[alias = e]| x^{---+} + x^{++-+} + x^{-+--} + x^{-+++} & \longmapsto & |[alias = E]| x^{-+-+} \\
                 |[alias = d]| x^{+++-} + x^{--+-} + x^{-+--} + x^{-+++} & \longmapsto & |[alias = D]| x^{-++-}
        \end{tikzcd}
    \end{center}

    This completes the case of $x$ being of type $\romanNumeralCaps{3}$. If $x$ is of type $\romanNumeralCaps{2}$, then the necessary verifications can be obtained from the lists already given above: using \cref{notation}, in each case a number of the points become identified, indicated by the dark blue boxes, and should be treated as a single point with dilation factor 2. We see that the bijections above are compatible with these identifications. Finally, if $x$ is of type $\romanNumeralCaps{1}$ the verification is similar and indicated by the light green boxes representing a single point with dilation factor 3.

    \medskip 
    At this point we have established all the bijections involved in the triality fibrewise. Note also that those bijections do not depend on the choice of coordinates because the rules defining the bijections do not depend on the given coordinates. To see that these are in fact isomorphisms of graphs, we need to check that they are compatible with the root maps and the involution on half-edges.

    For the first one, given a vertex $v\in K$, one can choose a coordinate system such that for any half-edge $h$ in the tangent space of $v$ we have that $r(h_i^{\pm})=v_i^{\pm}$, i.e. the coordinate system is trivial around $v$. With this coordinate system, it is clear that the fibrewise isomorphism preserves the root map.

    For the compatibility with the involution one can choose a coordinate system which is invariant under the involution of half-edges. With that coordinate system, it is clear that the fibrewise isomorphism preserves the involution of half-edges.

    In summary, the fibrewise isomorphism preserves the graph structure, hence is a graph isomorphism.
\end{proof}

When a tower is orientable, we know that $\tilde P_{\tilde G}$ is not connected. However it might have more than two connected components, producing new towers where the top graph is not connected. The next proposition provides connectedness of the $\tilde C_i$ for a large set of examples.

\begin{proposition} \label{prop:Donagi_connected}
    Let $\widetilde{G}\to G\to K$ be a generic tower (in particular let $\tilde G \to G$ be free) with connected $\widetilde{G}$ such that $\tilde G \to K$ is a fibrewise quotient of the trivial octuple cover of $K$. Then the tower is orientable and the output curves $\tilde C_1$ and $\tilde C_2$ are connected.
\end{proposition}

Note that, whenever $K$ is a tree, the map $\tilde G \to K$ is automatically a quotient of the trivial octuple cover of $K$.

\begin{proof}    
     We define $N = \set{1, 2, 3, 4, -1, -2, -3, -4}$. We name the projection $p:\Tilde{G}\to K$ and we let $q :\coprod_{N} K \to \widetilde{G}$ be the (non-harmonic) fibrewise quotient map. We define the group $WD_4$ to be the subgroup of $S_N$ generated by 
     $$(1,2)(-1,-2), \quad (1,3)(-1,-3), \quad (1,4)(-1,-4), \quad \text{and} \quad(1,-1)(2,-2).$$ 
     A direct computation shows that $|WD_4|=192$. This group is the even signed permutation group and already appears in \cite[Remark 2.6]{donagiSecond}.
    
    Now let $x$ be a point of $K$, we associate to it the subgroup of $WD_4$ of all the permutations (in $WD_4$) that act as identity on $p^{-1}(x)$. In more rigorous terms, $WD_4$ acts on $\coprod_{N} K$ as a subgroup of $S_N$. We associate to $x$ the subgroup of $WD_4$ of all permutations $s$ such that the restriction $\Bar{q}:\coprod_{N} \set{x}\to \Tilde{G}$ satisfies $\Bar{q}\circ s = \Bar{q}$. We name this subgroup $K_x$. 

    The cover $\Gtilde$ is connected if and only if the subgroup generated by all $K_x$ acts transitively on $N$. To see this, let us remark that a path inside $\Gtilde$ can be lifted to a union of paths inside $\coprod_{N} K$, such that paths begin and end at points that are the same under the quotient map $q$. An explicit computation fibre type per fibre type shows that all $K_x$ are always generated by bitranspositions. A direct computation in Magma shows that the only subgroup of $WD_4$ that acts transitively on $N$ and is generated by bitranspositions is $WD_4$ itself. Therefore, the subgroup generated by the union of all $K_x$ is $WD_4$.

    Now we can define $M$ to be set of 4 element subsets of $N$ such that the map $X\to \set{1,2,3,4}, a \mapsto |a|$ is bijective for all $X \in M$. The action of $WD_4$ on $M$ splits in exactly two orbits: the subsets with an even number of negative numbers and the subsets with an odd number of negative numbers. The elements of $M$ are exactly the quadruplets of divisors defined in the tetragonal construction, and the action of $WD_4$ on $M$ describes the connectedness of $\Ptilde_{\Gtilde}$ in the same way as the action of $WD_4$ on $N$ describes the connectedness of $\Gtilde$.
\end{proof}

It can also happen that the input tower is connected while the output tower is not. The following lemma gives a restriction on the possibilities of tetragonal links between connected and disconnected towers.

\begin{lemma}\label{lem:connected_donagi}

    Let $\Gtilde\to G\to K$ be an orientable generic tower. If $\Gtilde$ is disconnected and the output of the tetragonal construction are two towers $\tilde C_i\to C_i\to K$ satisfying that $\tilde C_i$ is connected for at least one $i\in\set{1,2}$, then $G$ is disconnected into two connected components $G=G_1\coprod G_2$ where the induced map $G_i\to K$ is a double cover for $i=1,2$.
\end{lemma}
\begin{proof}
    If $\Gtilde$ is disconnected, then either $G$ is connected or it is not.

    If $G$ is connected, then $\Gtilde\to G$ is the trivial double cover and we can label the sheets as $\Gtilde=G^+ \coprod G^-$ with $G^+\cong G^-\cong G$. One can define a map $\text{deg}_+:\Gtilde\to \Z$ mapping $x$ to 1 if $x\in G^+$ and 0 elsewhere. This map can be extended to a map $\text{deg}_+:\Ptilde_{\Gtilde}\to \Z$. The latter is locally constant and has image $\set{0,1,2,3,4}$. The points mapping to 0, 2 and 4 come from the component $\tilde C_1$ hence this component is disconnected and the points mapping to 1 and 3 come from the component $\tilde C_2$ hence that component is also disconnected and we are done.

    In the case where $G$ itself is disconnected as well, we can either decompose $G$ into two connected components as in the theorem or we can decompose it as $G=G'\coprod K$ where the map $G'\to K$ is degree 3 (with $G'$ not necessarily connected). We then can decompose $\Gtilde = \tilde G'\coprod K\coprod K$. Then any point $D\in \Ptilde_{\Gtilde}$ decomposes as $D' + y$ where $y$ is one point in the two copies of $K$ and $D'\in \Z^{\Gtilde'}$. Then $\Ptilde_{\Gtilde}$ is made of two copies of the trigonal construction $\Ptilde_{\Gtilde'}$ that are exchanged via the involution, hence four copies of a degree 4 cover of $K$. This implies that both towers of $\Ptilde_{\Gtilde}$ are disconnected.

    Therefore, if $\Gtilde$ is disconnected and $\tilde C_i$ is connected for at least one $i\in\set{1,2}$ then $G$ is disconnected and cannot be decomposed as $G=G'\coprod K$ where $G'\to K$ is a degree 3 cover. Hence, $G$ can only be decomposed into two connected components $G=G_1\coprod G_2$ where $G_1\to K$ and $G_2\to K$ are both double covers.
\end{proof}

Lastly, let us show an example of a triality of towers where two towers are connected and one tower is disconnected.

\begin{example}

    \input{disconnectedExample}
\end{example}

\begin{definition} \label{def:good_tower}
    A tower $\Gtilde \to G \to K$ is \emph{good} if it is generic, $\tilde G$ is connected, the tower is orientable and the tetragonal output is two connected towers. In particular, in any good tower $\Gtilde \to G$ is free.
\end{definition}

By \cref{lem:orientable_quotient} and \cref{prop:Donagi_connected}, we know that a class of examples of good towers is the class of generic towers such that the map $\Gtilde\to K$ is a quotient of the trivial octuple cover and $\Gtilde$ is connected.

Note that, for each fibre of generic type, the tetragonal construction gives two copies of the same type. This does not mean that the tetragonal construction is trivial since the global structure of the tower may change. 
However, the tetragonal construction preserves some properties of the tower: if a tower is a connected quotient of the trivial octuple cover, then it is orientable and the output towers of the tetragonal construction are also connected and quotients of the trivial octuple cover.
This is an extension of \cite[Proposition~2.14]{roehrleNgonal} to a larger class of towers where $K$ is not necessarily a tree.

\begin{proposition} \label{prop_dimensions}
    Let $\tilde G \to G \to K$ be a good tower and $\tilde C_i \to C_i \to K$ be the two output towers of the tetragonal construction. Then the dimensions of the Prym varieties agree, i.e. 
    \[ g(\Gtilde) - g(G) = g(\tilde C_1) - g(C_1) = g(\tilde C_2) - g(C_2). \]
\end{proposition}

\begin{proof}
    By definition, $\tilde G$, $G$, $\tilde C_i$, and $C_i$ are all connected.
    For a connected graph $H$, the genus is $g(H) = |E(H)| - |V(H)| + 1$. But restricted to any individual fibre, the tetragonal construction does not change its type. Therefore, $|E(\tilde G)| = |E(\tilde C_1)| = |E(\tilde C_2)|$ and $|E(G)| = |E(C_1)| = |E(C_2)|$ and similarly for the number of vertices. The claim follows immediately. 
\end{proof}

\section{The tropical Donagi theorem}
\label{sec-thm}

In this section we will prove the tropical Donagi theorem, \cref{thm:main}, and this section closely follows the ideas of \cite{izadiLange}. From now on we work with a \emph{good tower of metric graphs} $\tilde \Gamma \to \Gamma \to K$, meaning that the underlying tower of unmetrised graphs is good in the sense of \cref{def:good_tower}. 
The tetragonal construction of the metric tower is given by performing the tetragonal construction on the underlying unmetrised tower and then metrising the result in the unique way to make the morphisms harmonic maps of metric graphs. By the assumption on $\tilde \Gamma \to \Gamma \to K$ being a good tower, the result is two good metric towers which we denote $\tilde\Lambda_i \to \Lambda_i \to K$ with $i=1,2$. 

The fibrewise description of the Donagi construction from \cref{def:Donagi_construction} holds verbatim for the metric Donagi construction when \enquote{points} are taken to be points in the sense of a metric graph. In this setting, the free $\ZZ$-modules in \cref{def:Donagi_construction} may be replaced with $\Div(\tilde\Gamma)$. Recall that by definition of a good tower, all fibres in our setting are of one of the types from \cref{def:good_types}. Moreover, the double covers $\tilde \Gamma \to \Gamma$ and $\tilde \Lambda_i \to \Lambda_i$ are free and hence the Prym varieties in question are all naturally principally polarised and we may work with the description of the Prym by antisymmetric divisors.

Concerning \cref{thm:main}, \cref{prop:triality} already establishes the triality and by \cref{prop_dimensions} we know that the dimensions of the Prym varieties in question agree. The goal of this section is to prove that $\Prym(\tilde\Gamma / \Gamma) \cong \Prym(\tilde\Lambda_1 / \Lambda_1)$ in the sense of principally polarised tropical abelian varieties. The isomorphism between $\Prym(\tilde\Gamma / \Gamma)$ and $\Prym(\tilde\Lambda_2 / \Lambda_2)$ can be obtained similarly because formally relabelling $x_4^+$ into $x_4^-$ and vice versa allows one to pass from the odd component of the tetragonal construction to the even component and vice versa. Indeed, the relabelling of $x_4^\pm$ automatically relabels the point $x^{++++}$ into $x^{+++-}$ hence passes from an even number of $+$ to an odd number.
\medskip

We start the proof by introducing two maps. Let 
\begin{align*}
    S:\tilde\Lambda_1&\longrightarrow \Div^4(\Gam),\\ z =x_1^{\epsilon_1}+x_2^{\epsilon_2}+x_3^{\epsilon_3}+x_4^{\epsilon_4}&\longmapsto x_1^{\epsilon_1}+x_2^{\epsilon_2}+x_3^{\epsilon_3}+x_4^{\epsilon_4}
\end{align*} 
be the map which reinterprets a formal sum of four points as an actual divisor
and let 
\begin{align*}
    S^T : \Gam&\longrightarrow \mathrm{Div}^{4}(\tilde\Lambda_1), \\ 
    x&\longmapsto \sum_{\substack{z \in \tilde\Lambda_1 \\ z(x) > 0}} \frac{d_p(z) z(x)}{d_{f \circ \pi}(x)} z,
\end{align*}
where $\pi:\Gam\to\Gamma$, $f:\Gamma\to K$ and $p:\tilde\Lambda_1\to K$ are the covers.

The map $S^T$ can be spelt out in coordinates as follows. In a fibre of type $\romanNumeralCaps{3}$, we label the points of $\tilde \Gamma$ by $\{x_1^\pm, \ldots, x_4^\pm\}$ and obtain
\begin{equation*}
    S^T : x \longmapsto \sum x+x_2^{\epsilon_2}+x_3^{\epsilon_3}+x_4^{\epsilon_4} 
\end{equation*}
where the sum is taken over all triplets of signs $(\epsilon_2,\epsilon_3,\epsilon_4)$ such that $x+x_2^{\epsilon_2}+x_3^{\epsilon_3}+x_4^{\epsilon_4}$ lies in $\tilde\Lambda_1$. In a fibre of type $\romanNumeralCaps{2}$ or $\romanNumeralCaps{1}$ the same description holds true by using \cref{notation}, i.e. formally identifying $x_i = x_j$ and $x_i = x_j = x_k$, respectively. This map works the same way as $S$ does, but we compose it with the isomorphism between $\Gam$ and the tetragonal output of $\tilde\Lambda_1$.

Extending by linearity, $S$ and $S^T$ induce maps of divisors (hence group homomorphisms) $S:\Div^d(\tilde\Lambda_1)\to\Div^{4d}(\Gam)$ and $S^T:\Div^d(\Gam)\to\Div^{4d}(\tilde\Lambda_1)$ for any degree $d$. We restrict now to $d = 0$. 

\begin{lemma} \label{lem:function}
    Let $f : \tilde\Lambda_1 \to \RR$ be a continuous piecewise linear function, not necessarily rational and let 
    $$\tilde f:\Gam\longrightarrow\R, \qquad  x\longmapsto\sum_{\substack{z \in \tilde \Lambda_1 \\ z(x) > 0}} \frac{d(z)z(x)}{d(x)} f(z)$$ 
    be the pushforward of $f$ where $d(-)$ denotes dilation with respect to the structural map to $K$. Then $\tilde f$ is a continuous piecewise linear function.
\end{lemma}

\begin{proof}
    \input{proofOfContinuity}  
\end{proof}

\begin{lemma} \label{lem:slopes}
    In the setting of \cref{lem:function}, the derivative of $\tilde f$ at a given point $x$ is 
    \[ \tilde f'(x) = \sum_{\substack{z \in \tilde\Lambda_1 \\ z(x) > 0}} z(x) \cdot f'(z).  \]
\end{lemma}

\begin{proof}
    To compute the derivative of $\tilde f$, let $\tilde e = (\tilde v, \tilde w)$ be an edge of $\tilde \Gamma$ lying over the edge $e \in E(K)$. Since the slope of $\tilde f$ only needs to be checked on interior points of edges of $\tilde \Gamma$ and away from the support of $\div(f)$, we may assume without loss of generality that the combinatorial type of the fibre stays the same across $e$ and the slopes of $f$ and $\tilde f$ are constant over $e$. 
	Returning to the combinatorial description of the tetragonal construction, we define the following set of edges of $\tilde\Lambda_1$ over $e$
	\[\mathcal{E} = \big\{ u \in E(\tilde\Lambda_1) \mathrel{\big|} [S(u)](\tilde e) > 0 \big\}, \]
    where $[S(u)](\tilde e)$ means $[S(z)](x) = z(x)$ for any points $z\in u$ and $x\in \tilde e$ which lie above the same point of $K$.
    We fix an orientation on all the preimages of $e$ such that they map to the same orientation of $e$.
	With the assumption that dilation factors do not change over $e$ we compute:
	\begin{align*}
		\text{$\tilde f'$ on } \tilde e &= \frac{1}{\ell(\tilde e)} \big( \tilde f(\tilde w) - \tilde f(\tilde v) \big) \\
	   &= \frac{1}{\ell(\tilde e)} \sum_{u = (y, z) \in \mathcal{E}} \frac{d(z)z(\tilde w)}{d(\tilde w)} f(z)-\frac{d(y)y(\tilde v)}{d(\tilde v)} f(y)\\
        &= \frac{1}{\ell(\tilde e)d(\tilde e)} \sum_{u = (y, z) \in \mathcal{E}} {d(u) [S(u)](\tilde e)} \big( f(z) - f(y) \big) \\
        &= \frac{d(\tilde e)}{\ell(e)} \sum_{u = (y, z) \in \mathcal{E}} \frac{d(u) [S(u)](\tilde e)}{d(\tilde e)} \big( f(z) - f(y) \big) \\
		&= \sum_{u = (y, z) \in \mathcal{E}} [S(u)](\tilde e) \underbrace{\frac{d(u)}{\ell(e)} \big( f(z) - f(y) \big)}_{\text{slope of $f$ on } u}.
	\end{align*}
\end{proof}

\begin{lemma} \label{prop:pushforward_div}
	Let $D \in \Div^0(\tilde\Lambda_1)$ and write it as $D = \div(f)$ with $f : \tilde\Lambda_1 \to \RR$ a continuous piecewise linear function, not necessarily rational. Let $\tilde f : \tilde \Gamma \to \RR$ be the pushforward of $f$ defined in \cref{lem:function}.
	With this definition we have $S(D) = \div(\tilde f)$.
\end{lemma}

Let us remind that two piecewise linear functions give rise to the same divisor if and only if their difference is a constant function.

\begin{proof}
    For $x\in\tilde\Gamma$, we choose an orientation such that all edges incident to $x$ are going away from $x$ and we do the same for all other points of $\tilde\Gamma$ in the same fibre over $K$. We obtain: 
    \begin{align*}
        \div(\tilde{f})(x)=\sum_{\tilde e\in r^{-1}(x)}\tilde f'(\tilde e)&=\sum_{\tilde e\in r^{-1}(x)}\sum_{\substack{e\in E(\tilde\Lambda_1)\\ [S(e)](\tilde e)>0}}[S(e)](\tilde e)f'(e) \\ &=\sum_{\substack{z\in \tilde\Lambda_1\\z(x)>0}}z(x)\sum_{e\in r^{-1}(z)}f'(e)=\sum_{\substack{z\in \tilde\Lambda_1\\ z(x)>0}}z(x) D(z)=[S(D)](x) .  
    \end{align*}
\end{proof}

\begin{proposition}\label{prop:S_jacobian}
	The maps $S$ and $S^T$ induce group homomorphisms of Jacobians.
\end{proposition}
\begin{proof}
	One needs to check that $S$ and $S^T$ send principal divisors to principal divisors. 
	To this end let $D$ be a principal divisor on $\tilde\Lambda_1$ and write $D=\div(f)$ for a rational function $f:\tilde\Lambda_1\to \R $. By \cref{prop:pushforward_div} we know that $S(D) = \div(\tilde f)$ for $\tilde f$ as in \cref{lem:function} and it remains to check that $\tilde f$ has integer slopes. This follows immediately from \cref{lem:slopes}. 
    Similarly, for $F$ a principal divisor on $\tilde\Gamma$ we write $F = \div (g)$ for a rational function $g : \tilde\Gamma \to \RR$. Then we claim that $S^T(F) = \div(\tilde g)$, where $\tilde g : \tilde\Lambda_1 \to \RR$ is the rational function
    \[ z \longmapsto \sum_{x : z(x) > 0} z(x) \cdot g(x). \]
    The necessary verifications are analogous to the Lemmas~\ref{lem:function}, \ref{lem:slopes}, and~\ref{prop:pushforward_div}.
	Finally, linearity of $S$ and $S^T$ is clear from the definition of the maps by linear extension.
\end{proof}

Recall from Notation \ref{notation} the short-hand notation $x^{\pm\pm\pm\pm} = x_1^\pm + x_2^\pm + x_3^\pm + x_4^\pm$ for a point in $\tilde\Lambda_1$. 

\begin{lemma}\label{FormulaOneS}
    Let $g : \tilde \Gamma \to K$ and $p : \tilde\Lambda_1 \to K$ be the structural maps. Then we have
    \begin{align*}
        S^T \circ S(z) &= 2\big( p^\ast(p(z)) + z - \iota z \big) & &\text{for all } z\in \tilde\Lambda_1 \text{ and} \\
        S \circ S^T(x) &= 2 \big( g^\ast(g(x)) + x - \iota x \big) & &\text{for all } x \in \tilde \Gamma.
    \end{align*}
\end{lemma}

\begin{proof}
    Let $z \in \tilde\Lambda_1$ and assume that $p(z)$ is of type $\romanNumeralCaps{3}$. The computations for types $\romanNumeralCaps{2}$ and $\romanNumeralCaps{1}$ arise from the proof of type $\romanNumeralCaps{3}$ by formally identifying $x_i = x_j$ and $x_i = x_j = x_k$ for suitable $i, j, k$, respectively.     
    Without loss of generality we choose the labelling of the fibre in $\tilde \Gamma$ over $p(z)$ such that $z = x^{++++}$. With this we may compute
    $$S^T\circ S (x^{++++})=S^T(x_1^+)+S^T(x_2^+)+S^T(x_3^+)+S^T(x_4^+).$$ 
    Then, one computes individually: 
    \begin{align*}
        S^T(x_1^+) = x^{++++}+x^{++--}+x^{+-+-}+x^{+--+}\\
        S^T(x_2^+) = x^{++++}+x^{++--}+x^{-++-}+x^{-+-+}\\
        S^T(x_3^+) = x^{++++}+x^{+-+-}+x^{-++-}+x^{--++}\\
        S^T(x_4^+) = x^{++++}+x^{+--+}+x^{-+-+}+x^{--++}
    \end{align*}
    and so the formula for $S^T\circ S$ follows.
    For the second part we let $x \in \tilde \Gamma$ and label the fibre such that $x = x_1^+$. Then
    $$S\circ S^T(x)=S(x^{++++})+S(x^{++--})+S(x^{+-+-})+S(x^{+--+}).$$
    Then one computes individually: \begin{align*}
        S(x^{++++})= x + x_2^+ +x_3^+ +x_4^+\\
        S(x^{++--})= x + x_2^+ +x_3^- +x_4^-\\
        S(x^{+-+-})= x + x_2^- +x_3^+ +x_4^-\\
        S(x^{+--+})= x + x_2^- +x_3^- +x_4^+
    \end{align*}
    but knowing that $x_i^-=\iota x_i^+$, the result follows.
\end{proof}

For the following proposition we view the Prym variety as the locus of (classes of) antisymmetric divisors in the Jacobian. Since the double covers are all free, the Prym is really a subvariety of the Jacobian.
Let us denote $P = \Prym(\Gam/\Gamma)$ and $P_1 = \Prym(\tilde\Lambda_1/\Lambda_1).$

\begin{proposition} \label{prop:S_Prym}
    The maps of Jacobians $S$ and $S^T$ restrict to maps of Prym varieties $s : P_1 \to P$ and $s^t : P \to P_1$.
\end{proposition}
\begin{proof}
    An element of a Prym variety can always be chosen antisymmetric (see \cite[Lemma~5.9]{lenIntroductionToDivisors}). Hence it is sufficient to verify that the image of an antisymmetric divisor is indeed antisymmetric, which in turn can be achieved by showing that $S$ and $S^T$ commute with the respective involutions $\iota$. For $S$ this is obvious by the definition of $\iota$ on the output of the Donagi construction. 
    For $S^T$ we let $x\in \tilde\Gamma$ and assume first that the fibre which contains $x$ is of type $\romanNumeralCaps{3}$. We label the fibre with $\{x_1^\pm, \ldots, x_4^\pm\}$ such that $x = x_1^+$ and such that $\tilde\Lambda_1$ is the even part of the tetragonal construction. With this we compute
    \begin{align*}
        S^T(\iota x_1^+) = S^T(x_1^-) ={}& x^{----} + x^{-++-} + x^{-+-+} + x^{--++}\\
        ={}& \iota \big(x^{++++}+x^{+--+}+x^{+-+-}+x^{++--} \big)
        = \iota S^T(x_1^+)
    \end{align*}
    and this verifies the claim for a fibre of type $\romanNumeralCaps{3}$. The verifications for types $\romanNumeralCaps{2}$ and $\romanNumeralCaps{1}$ are obtained from this computation by formally identifying suitable $x_i = x_j$ and $x_i = x_j = x_k$, respectively.
\end{proof}

\begin{corollary} \label{cor:4Id}
    We have $s^t\circ s = 4\cdot \operatorname{Id}_{P_1}$ and $s\circ s^t = 4\cdot \mathrm{Id}_{P}$.
\end{corollary}

\begin{proof}
    Let $p: \tilde \Lambda_1 \to K$ be the structural map. Using the first formula of Lemma~\ref{FormulaOneS}, we obtain 
    \begin{align*}
        s^t \circ s(z - \iota z) &= 2 \big( p^\ast(p(z)) + z - \iota z \big) - 2\big( p^\ast(p(z)) + \iota z - z \big) \\
        &= 4(z - \iota z)
    \end{align*}
    for any $z \in \tilde\Lambda_1$. Since any $D - \iota D \in P_1$ is a formal sum of elements of the form $z - \iota z$, this shows the first part. The second part follows in the exact same fashion using the second formula of \cref{FormulaOneS}.
\end{proof}

\begin{proposition}\label{polarization}
    The maps $s$ and $s^t$ double the induced polarisations of the Prym varieties.
\end{proposition}

\begin{proof}

	Let $D -\iota D\in\Div(\Gam)$ and let $f$ (not necessarily rational) satisfy $\div(f)=D-\iota D$. Possibly after simultaneously refining the models of $\tilde \Gamma$, $\Gamma$, and $K$ we may assume that $D-\iota D$ is supported on vertices. We fix an orientation on $K$ and consider $\tilde \Gamma$ and $\Gamma$ with the compatible orientations. Denote $e_i$ the slope of $f$ at any point of the $i$-th preimage of the edge $e$ (the antisymmetry makes that we do not need to differentiate the slope for the two preimages of $e_i$ because their squares will be the same). As of now, we assume that every edge has dilation profile $\romanNumeralCaps{3}$. Evaluating the polarisation from \cref{eq:pol_Pic} yields
	\begin{equation}\label{normOne}
	    ||D-\iota D||^2=\sum_{e\in E(K)}2 \ell(e) \cdot (e_1^2+e_2^2+e_3^2+e_4^2).
	\end{equation}
    Then, using \cref{lem:slopes} one establishes that 
    \begin{equation} \label{normTwo}
        \begin{aligned}
            ||s^t(D-\iota D)||^2=&\sum_{e\in E(K)}2 \ell(e) \cdot \Big((e_1+e_2+e_3+e_4)^2+(e_1+e_2-e_3-e_4)^2 \\
            &+(e_1-e_2+e_3-e_4)^2+(e_1-e_2-e_3+e_4)^2 \Big) \\
            =&\sum_{e\in E(K)}2 \ell(e) \cdot(4e_1^2+4e_2^2+4e_3^2+4e_4^2) \\
            ={}& 4||D-\iota D||^2.
        \end{aligned}
    \end{equation}
    This shows that the Euclidean norm is doubled when applying the map $s^t$.

	Now we consider the other types of dilation profiles. Assume that an edge $e$ has type $\romanNumeralCaps{2}$. Without loss of generality, we assume that the first and second preimage of the edge $e$ are the same and we write $e_1 = e_2$ as half of the slope on that edge. Again by \cref{lem:slopes} the term corresponding to $e$ in the sum of Equation~\eqref{normOne} becomes 
	$$ 2\frac{\ell(e)}{2} \cdot (2e_1)^2+2\ell(e)(e_3^2+e_4^2) =2\ell(e) \cdot(e_1^2+e_2^2+e_3^2+e_4^2). $$ 
	Therefore, with dilation profiles of type $\romanNumeralCaps{2}$, Equation~\eqref{normOne} remains unchanged. The proof is similar for dilation profile $\romanNumeralCaps{1}$. Therefore, Formula~\eqref{normOne} does not depend on the dilation profiles.

	Now we prove that Formula~\eqref{normTwo} also does not depend on the dilation profiles. Assume that the edge $e$ has dilation profile of type $\romanNumeralCaps{2}$. Under the same notation as before, the corresponding term in Equation~\eqref{normTwo} becomes 
	\begin{align*}
        &2\ell(e) \big((e_1+e_2+e_3+e_4)^2+(e_1+e_2-e_3-e_4)^2 \big) \\
        &+ \ell(e) \big((e_1-e_2+e_3-e_4)^2+(e_1-e_2-e_3+e_4)^2 \big)\\
        ={} &2\ell(e)(4e_1^2+4e_2^2+4e_3^2+4e_4^2).
    \end{align*}
	The proof is similar for the dilation profile  $\romanNumeralCaps{1}$. Hence it follows that $s^t$ doubles the induced polarisation. The statement for $s$ follows from \cref{cor:4Id}.
\end{proof}

The results of this section so far can be translated into the language of lattices as follows:

\begin{corollary}
    The correspondences $S$ and $S^T$ induce maps
    \[ S : H_1(\tilde\Lambda_1, \ZZ) \longrightarrow H_1(\tilde\Gamma, \ZZ) \qquad \text{ and } \qquad
            S^T : H_1(\tilde \Gamma, \ZZ) \longrightarrow H_1(\tilde\Lambda_1, \ZZ). \]
    These maps descend to Prym varieties. More specifically, let $L = \Ker \pi_\ast$ the lattice defining $\Prym(\tilde\Gamma / \Gamma)$ of $\pi : \tilde\Gamma \to \Gamma$ and $L_1$ the lattice defining $\Prym(\tilde \Lambda_1 / \Lambda_1)$. Then we have $S\vert_{L_1} : L_1 \to L$ and $S^T\vert_{L} : L \to L_1$.
\end{corollary}

\begin{proof}    
    The polarisation can be considered a pairing on the space of formal sums of edges. Inside this space, the lattice $H_1$ is the orthogonal complement of the space of piecewise linear functions.
    The maps $s$ and $s^t$ considered on the space of formal sums of edges double the polarisation by \cref{polarization} and leave the space of piecewise linear functions invariant by \cref{prop:S_jacobian}. Hence they also preserve the $H_1$ lattices. 
    The restriction to the lattices $L$ and $L_1$ follows from \cref{prop:S_Prym}.
\end{proof}

We now need to add the hypothesis that $K$ is a tree. Up until now, this was not necessary.

\begin{proposition} \label{prop:divide_by_2}
    Let $\tilde\Gamma \to \Gamma \to K$ be a tower where $K$ is a tree, and $\tilde\Lambda_1 \to \Lambda_1 \to K$ the even tetragonal output. 
    We claim that $S : L_1 \to L$ and $S^T : L \to L_1$ both factor as an isomorphism followed by multiplication by 2.

    We name $\psi: L_1\to L$ the isomorphism in the factorisation of $S$.
\end{proposition}
\begin{proof}
    Since $K$ is a tree, for any fixed edge $e\in E(K)$ we consider the fibre over $e$ in $E(\tilde \Lambda_1)$. This fibre can be partitioned as $E^+\cup E^-$ such that $\iota(E^+) = E^-$. For a given cycle $\gamma =\sum_{e\in E(\tilde\Lambda_1)}a_ee\in H_1(\tilde\Lambda_1,\Z)$, we have that 
    $$\sum_{e\in E^+}a_e \equiv \sum_{e\in E^-}a_e \mod 2$$ 
    hence the cycle $\gamma -\iota \gamma$ has an even sum of coefficients in $E^+$. 
    
    \textbf{Case: $e$ is of type \romanNumeralCaps{3}.} Since any element in $P_1$ can be written as $\gamma - \iota \gamma$, we know that the part with edges in $E^+$ of any element of $P_1$ is in the $\ZZ$-span of the columns of
        \[ U = \begin{pmatrix} 
            2 & 1 & 0 & 0  \\
            0 & -1 & 1 & 0  \\
            0 & 0 & -1 & 1  \\
            0 & 0 & 0 & -1  \end{pmatrix}.  \]
    \noindent
    Now the part of $S(\gamma - \iota\gamma)$ in $E^+$ is given by multiplication by \[
        A = \begin{pmatrix}
        1 & 1 & 1 & 1 \\
        1 & 1 & -1 & -1 \\
        1 & -1 & 1 & -1 \\
        1 & -1 & -1 & 1
    \end{pmatrix}
    \] and we see that $A$ acting on the subspace $\langle U \rangle \subseteq \ZZ^{E^+}$
    \[ AU = 2U \underbrace{\begin{pmatrix}
            2 & 1 & 0 & 0 \\
            -3 & -2 & 0 & 0 \\
            -2 & -2 & 1 & 0 \\
            -1 & -1 & 0 & 1
        \end{pmatrix}}_{\text{invertible over $\ZZ$}} .  \]    

    \textbf{Case: $e$ is of type \romanNumeralCaps{2}.} In this case, $E^+$ has only three elements, and the matrices $U$ and $A$ take the shape
        \[ U = \begin{pmatrix} 
            2 & 1 & 0   \\
            0 & -1 & 1   \\
            0 & 0 & -1   \end{pmatrix} \qquad \text{ and } \qquad
        A = \begin{pmatrix}
        1 & 1 & 1  \\
        1 & 1 & -1  \\
        2 & -2 & 0 
    \end{pmatrix}
    \] and we see
    \[ AU = 2U \underbrace{\begin{pmatrix}
            2 & 1 & 0  \\
            -3 & -2 & 0   \\
            -2 & -2 & 1   
        \end{pmatrix}}_{\text{invertible over $\ZZ$}} .  \]

    \textbf{Case: $e$ is of type \romanNumeralCaps{1}.} This time the matrix presentations $A$ and $U$ take the shape
    \[ U = \begin{pmatrix} 
        2 & 1   \\
        0 & -1 
        \end{pmatrix} \qquad \text{ and } \qquad
        A = \begin{pmatrix}
        1 & 1  \\
        3 & -1   
    \end{pmatrix}
    \] and we see
    \[ AU = 2U \underbrace{\begin{pmatrix}
            2 & 1  \\
            -3 & -2  
        \end{pmatrix}}_{\text{invertible over $\ZZ$}} .  \]
    This concludes the proof that $S$ factors as multiplication by 2 followed by an isomorphism. By \cref{cor:4Id} the same holds for $S^T$.   
\end{proof}

We remark that the proof of \cref{prop:divide_by_2} only holds at the level of Prym varieties, as otherwise one does not always have $\sum_{e\in E^+}a_e\equiv 0\mod 2$ for an arbitrary cycle.

\begin{example}
    \input{treeNecessary}

\end{example}

The upshot of \cref{prop:divide_by_2} is that the isomorphism $\psi$ gives the desired isomorphism of abelian varieties and the following completes the proof of our Main \cref{thm:main}.

\begin{theorem}
    The map $\psi:P_1\to P$ constructed in \cref{prop:divide_by_2} is an isomorphism of principally polarised abelian varieties.
\end{theorem}
\begin{proof}
    By construction in \cref{prop:divide_by_2}, the map $\psi$ is an isomorphism of the lattices defining the Prym varieties. By \cref{polarization} it preserves the polarisation. Hence it is an isomorphism of principally polarised tropical abelian varieties.
\end{proof}

\section{Examples}
\label{sec:examples}

In this section we present two examples illustrating \cref{thm:main}. For the pictures we draw, we will use the following conventions throughout. First, thickness encodes the dilation factor with respect to the base graph $K$. Second, in any tower $\tilde \Gamma \to \Gamma \to K$ we indicate the double cover $\tilde \Gamma \to \Gamma$ already on the level of $\Gamma$ using \emph{signed graph notation}. This means that $\tilde \Gamma$ can be constructed from the picture by taking two disjoint copies of $\Gamma$ and then changing the preimages of the dashed edges of $\Gamma$ to connect the two sheets of $\tilde \Gamma$ rather than staying within the same sheet. 
Finally, the double cover $\tilde \Gamma \to \Gamma$ will always be understood as mirroring along the vertical axis of symmetry of $\tilde \Gamma$.

\input{firstExample}

\input{secondExample}

\section{The tropical \texorpdfstring{$n$}{n}-gonal construction in moduli}
\label{sec:non-cont}

In this section we show that the tropical $n$-gonal construction does not behave well in families and we provide evidence that \cref{thm:main} cannot be extended beyond its current formulation.
The \emph{moduli space of tropical curves} $M_g^\trop$ continuously parametrises all tropical curves of fixed genus $g$. For the construction it is important to allow non-trivial vertex genus functions on tropical curves as defined in \cref{sec-preliminaries}. In $M_g^\trop$, the combinatorial type of a tropical curve changes in the limit, when the length of an edge $e$ tends to zero. The combinatorial type in the limit is the contraction $G / e$. Perhaps more complicated is the description of the moduli space of principally polarised tropical abelian varieties $A_g^\trop$ of dimension $g$, see \cite{Brannetti} and \cite{ChanTropicalTorelli} for details. When set up correctly, one can make sense of the \emph{tropical Torelli morphism} $\Jac : M_g^\trop \to A_g^\trop$ as a continuous map of tropical moduli spaces, which is defined by associating to a tropical curve $\Gamma$ its Jacobian $\Jac(\Gamma)$.

A similar construction for Prym varieties, the \emph{tropical Prym\textendash Torelli morphism} is desirable. The moduli space of \emph{unramified} double covers $R_g^\trop$, with $g$ referring to the genus of the target curve, can be defined as a special case of a Hurwitz space in the sense of \cite{CavalieriMarkwigRanganathan}. In this space, a transition between combinatorial types happens when an edge in the target graph together with its preimage(s) are contracted simultaneously. Because of the polarisation issue for Prym varieties of dilated double covers, the correct definition of the tropical Prym\textendash Torelli morphism $\Prym : R_g^\trop \to A_{g-1}^\trop$ should use the continuous Prym varieties introduced in \cite[Definition~4.17]{roehrleNgonal} in order to give a continuous morphism. 

With this setup, \cref{thm:main} provides evidence for the non-injectivity of the Prym\textendash Torelli morphism. It is natural to ask, if this structure is modular in nature. More precisely, it is straightforward to set up a tropical moduli space $T_{2, n}^\trop$ parametrising towers $\tilde \Gamma \to \Gamma \to K$ where both maps are tropical Hurwitz covers, $\deg(\Gamma \to K) = n$, and  $K$ is a tree. While the $n$-gonal construction applied pointwise does give a map
\[ T_{2, n}^\trop \longrightarrow T_{2, 2^{n-1}}^\trop, \]
this map is not in general continuous under deformations which change the combinatorial type, see \cref{fig:Donagi_not_cont} for an example with $n = 2$.

\begin{figure}
    \centering
    \begin{tikzpicture}[scale = 0.9]
    \begin{scope}
        \draw (0,0) -- (0, 2);
        \vertex{0,0}
        \vertex{0,1}
        \vertex{0,2}
        \node[anchor = east] at (-0.2, 1.5) {$e$};
    
        \draw[->, black, thick] (1.6, 1) -- (0.4, 1);
    
        \draw[line width = 2pt] (2, 0) -- (2, 1);
        \draw (2, 1) -- (2, 2);
        \draw (2, 1) to[out = 50, in = -50] (2, 2);
        \vertex[2]{2, 0}
        \vertex[2]{2, 1}
        \vertex[2]{2, 2}
    
        \draw[->, black, thick] (3.6, 1) -- (2.4, 1);
    

        \draw[line width = 2pt] (4, 0) -- (4, 1);
        \draw (4, 1) -- (5, 2);
        \draw[line width = 2pt] (5, 0) -- (5, 1);
        \draw (5, 1) -- (4, 2);
        \draw (4, 1) to[out = 130, in = -130] (4, 2);
        \draw (5, 1) to[out = 50, in = -50] (5, 2);
        \foreach \i in {0, 1, 2}{
            \vertex[2]{4, \i}
            \vertex[2]{5, \i}
        }
    \end{scope}

    \draw[->, thick, decorate,decoration=snake] (6, 1) -- (7.6, 1);
    \node[anchor = south, align = center] at (6.8, 1.2) {bigonal \\ construction};

    \draw[->, thick, decorate,decoration=snake] (2, -1) -- (2, -3);
    \node[anchor = east] at (1.9, -2) {contract $e$};

    \draw[->, thick, decorate,decoration=snake] (11, -1) -- (11, -2);
    \node[anchor = west] at (11.1, -1.5) {contract $e$};

    \begin{scope}[xshift = 250]
        \foreach \i in {0, 2, 2.5}{
            \draw (\i,0) -- (\i, 2);
            \vertex{\i,0}
            \vertex{\i,1}
            \vertex{\i,2}
        }
        
        \draw[->, black, thick] (1.6, 1) -- (0.4, 1);
    
        \draw[->, black, thick] (4.1, 1) -- (2.9, 1);
    
        \draw[line width = 2pt] (4.5, 0) -- (4.5, 1);
        \draw (4.2, 2) -- (4.5, 1) -- (4.8, 2);
        \draw (5, 0) -- (5, 1) -- (5.3, 2) -- (5.6, 1) -- (5.6, 0);
        \vertex[2]{4.5, 0}
        \vertex[2]{4.5, 1}
        \vertex{4.2, 2}
        \vertex{4.8, 2}
        \vertex{5, 0}
        \vertex{5, 1}
        \vertex[2]{5.3, 2}
        \vertex{5.6, 1}
        \vertex{5.6, 0}
    \end{scope}

    \begin{scope}[yshift = -150]
        \vertex{0, 0}
        \draw (0,0) -- (0, 1);
        \vertex{0, 1}

        \draw[->, black, thick] (1.6, 0.5) -- (0.4, 0.5);        

        \vertex[2]{2, 0}
        \draw[line width = 2pt] (2, 0) -- (2, 1);
        \vertex[2]{2, 1}

        \draw[->, black, thick] (3.6, 0.5) -- (2.4, 0.5);

        \draw[line width = 2pt] (4, 0) -- (4.2, 1) -- (4.4, 0);
        \vertex[2]{4, 0}
        \vertex[4]{4.2, 1}
        \vertex[2]{4.4, 0}

        \draw[->, thick, decorate,decoration=snake] (6, 0.5) -- (7.6, 0.5);
        \node[anchor = south, align = center] at (6.8, 0.7) {bigonal \\ construction};
    \end{scope}

    \begin{scope}[xshift = 250, yshift = -100]
        \vertex{0, 0} 
        \draw (0, 0) -- (0, 1);
        \vertex{0, 1}

        \draw[->, black, thick] (1.6, 0.5) -- (0.4, 0.5);        

        \vertex{2, 0}
        \draw (2, 0) -- (2,1);
        \vertex{2, 1}
        \vertex{2.5, 0}
        \draw (2.5, 0) -- (2.5, 1);
        \vertex{2.5, 1}

        \draw[->, black, thick] (4.1, 0.5) -- (2.9, 0.5);

        \vertex[2]{4.5, 0}
        \draw[line width = 2pt] (4.5, 0) -- (4.5, 1);
        \vertex[2]{4.5, 1}

        \vertex{4.8, 0}
        \draw (4.8, 0) -- (5.1, 1) -- (5.4, 0);
        \vertex[2]{5.1, 1}
        \vertex{5.4, 0}
    \end{scope}

    \begin{scope}[xshift = 250, yshift = -150]
        \vertex{0, 0}
        \draw (0, 0) -- (0, 1);
        \vertex{0, 1}

        \draw[->, black, thick] (1.6, 0.5) -- (0.4, 0.5);        

        \vertex{2, 0}
        \draw (2, 0) -- (2.25, 1) -- (2.5, 0);
        \vertex[2]{2.25, 1}
        \vertex{2.5, 0}

        \draw[->, black, thick] (4.1, 0.5) -- (2.9, 0.5);

        \vertex[4]{5, 1}
        \draw[line width = 2pt] (5, 1) -- (5, 0);
        \vertex[2]{5, 0}
        \draw (4.5, 0) -- (5, 1) -- (5.5, 0);
        \vertex{4.5, 0}
        \vertex{5.5, 0}
    \end{scope}
    
\end{tikzpicture}
    \caption{The Donagi construction does not commute with edge contractions.}
    \label{fig:Donagi_not_cont}
\end{figure}
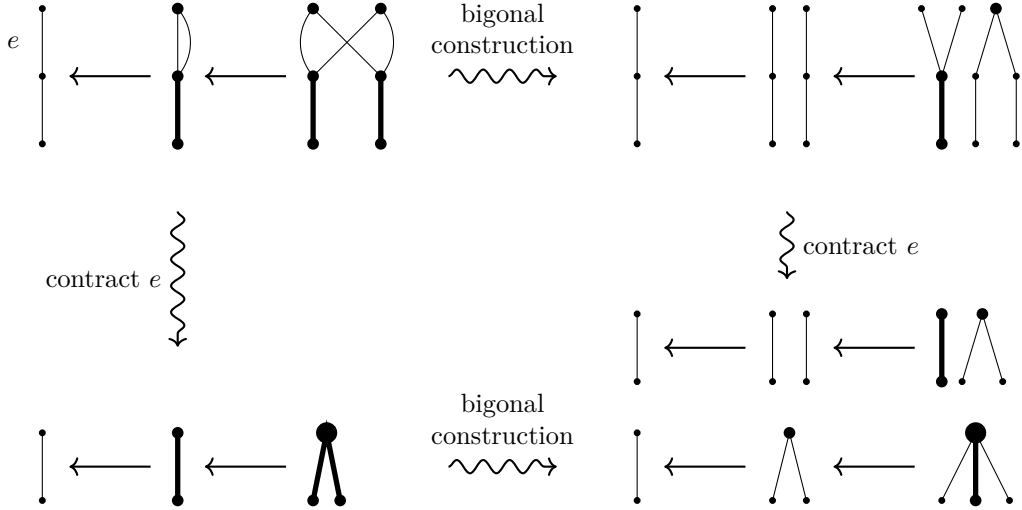

One could ask, if the definition of the $n$-gonal construction can perhaps be modified to improve this situation. More specifically, returning to the case $n = 4$, we take \cref{thm:main} as evidence, that the tetragonal construction is correctly defined on the locus of good towers inside of $T_{2, 4, h}^\trop$ and indeed, restricted to this locus the tetragonal construction gives rise to a commutative diagram of continuous maps
\begin{equation}
    \begin{tikzcd}[column sep = huge]
        T_{2, 4}^\trop \arrow[r, "\text{Donagi}"] \arrow[d, "\Prym"] & \operatorname{Sym}^2 T_{2, 4}^\trop \arrow[d, "\operatorname{Sym}^2 \Prym"] \\
        A_{g-1}^\trop \arrow[r, "A \mapsto A + A"] & \operatorname{Sym}^2 A_{g-1}^\trop.
    \end{tikzcd}
\end{equation}
Since the locus of good towers is not closed in $T_{2, 4}^\trop$, we ask for an extension of the tetragonal construction to the closure of this locus by continuity. It turns out that even this is impossible. 
Indeed, in \cref{fig:discont_example} we give an example of two covers of the same tree that will be the same after contracting a well-chosen edge but their tetragonal constructions will differ. In the picture, each of the two rows shows three towers related under the tetragonal construction. More precisely, each of the six graphs maps to the tree with edges shown on the left via a harmonic map of degree 4 and for each of the graphs we specify a free double cover in signed graph notation.

\begin{figure}[ht]
    \centering
    \begin{tikzpicture}[scale = 0.9]
    \draw[very thick] (-3, -1) -- (11.5, -1);

    \begin{scope}
        \edge{-2, 0}{-2, 5}
        \vertex{-2, 1}
        \vertex{-2, 2}
        \vertex{-2, 3}
        \vertex{-2, 4}
        \node[anchor = east] at (-2.2, 0.5) {$d$};
        \node[anchor = east] at (-2.2, 1.5) {$c$};
        \node[anchor = east] at (-2.2, 2.5) {$e$};
        \node[anchor = east] at (-2.2, 3.5) {$b$};
        \node[anchor = east] at (-2.2, 4.5) {$a$};
        \draw[->, black, thick] (-0.4, 2.5) -- (-1.6, 2.5);
    
        \edge{0, 0}{0.5, 1}
        \edge{0.5, 1}{0, 2}
        \edge{0, 2}{0, 3}
        \edge{0, 3}{0.5, 4}
        \edge{0.5, 4}{0, 5}

        \edge{1, 0}{0.5, 1}
        \edge{0.5, 1}{1, 2}
        \edge{1, 2}{1, 3}
        \edge{1, 3}{0.5, 4}
        \edge{0.5, 4}{1, 5}

        \vertex[2]{0.5, 1}
        \vertex[2]{0.5, 4}
        \vertex[2]{1, 2}
        \vertex[2]{1, 3}

        \edge{2, 0}{1.5, 1}
        \edge{1.5, 1}{1, 2}
        \edge{1, 3}{1.5, 4}
        \edge{1.5, 4}{2, 5}
        \foreach \y in {0, 1, 2, 3, 4}{
            \edge{2, \y}{2, \y + 1}
        }

        \draw[dashed] (1, 2) to[bend right = 50] (1, 3);

        \vertex[2]{2, 0}
        \vertex[2]{2, 5}       
    \end{scope}

    \begin{scope}[xshift = 120]
        \draw[dashed] (0.5, 0) -- (0, 1);
        \edge{0, 1}{0, 2}
        \edge{0, 2}{0.5, 3}
        \edge{0.5, 3}{0, 4}
        \edge{0, 4}{0, 5}

        \edge{0.5, 0}{1, 1}
        \edge{1, 1}{1.5, 2}
        \edge{1.5, 2}{1.5, 3}
        \edge{1.5, 3}{1, 4}
        \edge{1, 4}{1, 5}

        \edge{1, 0}{1.5, 1}
        \edge{1.5, 1}{1, 2}
        \edge{1, 2}{0.5, 3}
        \edge{0.5, 3}{1, 4}

        \edge{2, 0}{1.5, 1}
        \edge{1.5, 1}{1.5, 2}
        \edge{1.5, 2}{2, 3}
        \edge{2, 3}{2, 4}
        \edge{2, 4}{2, 5}

        \draw[dashed] (1, 4) to[bend right = 50] (1, 5);
        
        \vertex[2]{0.5, 0}
        \vertex[2]{0.5, 3}
        \vertex[2]{1, 4}
        \vertex[2]{1, 5}
        \vertex[2]{1.5, 1}
        \vertex[2]{1.5, 2}
    \end{scope}

    \begin{scope}[xshift = 240, yscale = -1, yshift = -142]
        \draw[dashed] (0.5, 0) -- (0, 1);
        \edge{0, 1}{0, 2}
        \edge{0, 2}{0.5, 3}
        \edge{0.5, 3}{0, 4}
        \edge{0, 4}{0, 5}

        \edge{0.5, 0}{1, 1}
        \edge{1, 1}{1.5, 2}
        \edge{1.5, 2}{1.5, 3}
        \edge{1.5, 3}{1, 4}
        \edge{1, 4}{1, 5}

        \edge{1, 0}{1.5, 1}
        \edge{1.5, 1}{1, 2}
        \edge{1, 2}{0.5, 3}
        \edge{0.5, 3}{1, 4}

        \edge{2, 0}{1.5, 1}
        \edge{1.5, 1}{1.5, 2}
        \edge{1.5, 2}{2, 3}
        \edge{2, 3}{2, 4}
        \edge{2, 4}{2, 5}    

        \draw[dashed] (1, 4) to[bend right = 50] (1, 5);
        
        \vertex[2]{0.5, 0}
        \vertex[2]{0.5, 3}
        \vertex[2]{1, 4}
        \vertex[2]{1, 5}
        \vertex[2]{1.5, 1}
        \vertex[2]{1.5, 2}
    \end{scope}

    \begin{scope}[yshift = -200]
        \edge{-2, 0}{-2, 5}
        \vertex{-2, 1}
        \vertex{-2, 2}
        \vertex{-2, 3}
        \vertex{-2, 4}
        \node[anchor = east] at (-2.2, 0.5) {$d$};
        \node[anchor = east] at (-2.2, 1.5) {$c$};
        \node[anchor = east] at (-2.2, 2.5) {$e$};
        \node[anchor = east] at (-2.2, 3.5) {$b$};
        \node[anchor = east] at (-2.2, 4.5) {$a$};
        \draw[->, black, thick] (-0.4, 2.5) -- (-1.6, 2.5);
        
        \edge{0, 0}{0.5, 1}
        \edge{0.5, 1}{0, 2}
        \edge{0, 2}{0, 3}
        \edge{0, 3}{0.5, 4}
        \edge{0.5, 4}{0, 5}

        \edge{1, 0}{0.5, 1}
        \edge{0.5, 1}{1, 2}
        \edge{1, 2}{1, 3}
        \draw[dashed] (1, 3) -- (0.5, 4);
        \edge{0.5, 4}{1, 5}

        \vertex[2]{0.5, 1}
        \vertex[2]{0.5, 4}
        \vertex[2]{1, 2}
        \vertex[2]{1, 3}

        \edge{2, 0}{1.5, 1}
        \edge{1.5, 1}{1, 2}
        \edge{1, 3}{1.5, 4}
        \edge{1.5, 4}{2, 5}
        \foreach \y in {0, 1, 2, 3, 4}{
            \edge{2, \y}{2, \y + 1}
        }

        \draw[dashed] (1, 2) to[bend right = 50] (1, 3);

        \vertex[2]{2, 0}
        \vertex[2]{2, 5}               
    \end{scope}

    \begin{scope}[yshift = -200, xshift = 120]
        \draw[dashed] (0.5, 0) -- (0, 1);
        \edge{0, 1}{0, 2}
        \edge{0, 2}{0.5, 3}
        \edge{0.5, 3}{0, 4}
        \draw[dashed] (0, 4) -- (0.5, 5);

        \edge{0.5, 0}{1, 1}
        \edge{1, 1}{1.5, 2}
        \edge{1.5, 2}{1.5, 3}
        \edge{1.5, 3}{1, 4}
        \edge{1, 4}{0.5, 5}

        \edge{1, 0}{1.5, 1}
        \edge{1.5, 1}{1, 2}
        \edge{1, 2}{0.5, 3}
        \edge{0.5, 3}{1.5, 4}
        \edge{1.5, 4}{1, 5}

        \edge{2, 0}{1.5, 1}
        \edge{1.5, 1}{1.5, 2}
        \edge{1.5, 2}{2, 3}
        \edge{2, 3}{1.5, 4}
        \edge{1.5, 4}{2, 5}    
        
        \vertex[2]{0.5, 0}
        \vertex[2]{0.5, 3}
        \vertex[2]{0.5, 5}
        \vertex[2]{1.5, 1}
        \vertex[2]{1.5, 2}
        \vertex[2]{1.5, 4}
    \end{scope}

    \begin{scope}[yshift = -200, xshift = 240]
        \edge{0, 0}{0, 1}
        \edge{0, 1}{0.5, 2}
        \edge{0.5, 2}{0, 3}
        \edge{0, 3}{0, 4}
        \edge{0, 4}{0, 5}

        \edge{1, 0}{1, 1}
        \edge{1, 1}{0.5, 2}
        \edge{0.5, 2}{1, 3}
        \edge{1, 3}{1, 4}
        \edge{1, 4}{1, 5}

        \draw[dashed] (1, 0) to[bend right = 50] (1, 1);
        \edge{1, 1}{1, 2}
        \edge{1, 2}{1.5, 3}
        \edge{1.5, 3}{1, 4}
        \draw[dashed] (1, 4) to[bend right = 50] (1, 5);

        \edge{2, 0}{2, 1}
        \edge{2, 1}{2, 2}
        \edge{2, 2}{1.5, 3}
        \edge{1.5, 3}{2, 4}
        \edge{2, 4}{2, 5}

        \vertex[2]{0.5, 2}
        \vertex[2]{1, 0}
        \vertex[2]{1, 1}
        \vertex[2]{1, 4}
        \vertex[2]{1, 5}
        \vertex[2]{1.5, 3}
    \end{scope}
\end{tikzpicture}
    \caption{Two triples of good tower related under the tetragonal construction.}
    \label{fig:discont_example}
\end{figure}
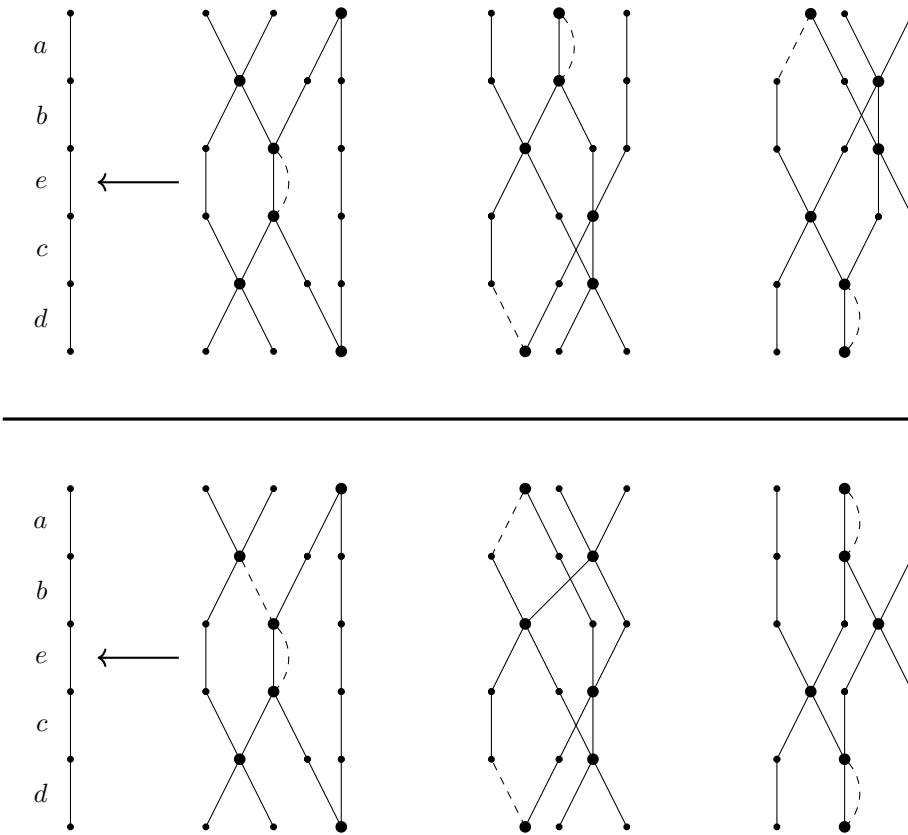

After contracting the edge $e$ and all of its preimages in all six towers, we obtain the towers shown in \cref{fig:discont_example_contracted}. The first tower is the same in both rows and here the double cover is no longer free (and hence the contracted tower is no longer good): the double cover is dilated over the orange vertex. The other four towers are not good either, but they are pairwise different.

\begin{figure}[ht]
    \centering
    \begin{tikzpicture}[scale = 0.9]
    \draw[very thick] (-3, -1) -- (11.5, -1);

    \begin{scope}
        \edge{-2, 0}{-2, 4}
        \vertex{-2, 1}
        \vertex{-2, 2}
        \vertex{-2, 3}
        \node[anchor = east] at (-2.2, 0.5) {$d$};
        \node[anchor = east] at (-2.2, 1.5) {$c$};
        \node[anchor = east] at (-2.2, 2.5) {$b$};
        \node[anchor = east] at (-2.2, 3.5) {$a$};
        \draw[->, black, thick] (-0.4, 2) -- (-1.6, 2);
    
        \edge{0, 0}{0.5, 1}
        \edge{0.5, 1}{0, 2}
        \edge{0, 2}{0.5, 3}
        \edge{0.5, 3}{0, 4}

        \edge{1, 0}{0.5, 1}
        \edge{0.5, 1}{1, 2}
        \edge{1, 2}{0.5, 3}
        \edge{0.5, 3}{1, 4}

        \vertex[2]{0.5, 1}
        \vertex[2]{0.5, 3}

        \edge{2, 0}{1.5, 1}
        \edge{1.5, 1}{1, 2}
        \edge{1, 2}{1.5, 3}
        \edge{1.5, 3}{2, 4}
        \foreach \y in {0, 1, 2, 3}{
            \edge{2, \y}{2, \y + 1}
        }

        \vertex[2]{2, 0}
        \vertex[2]{2, 4}       

        \filldraw[RedOrange, thick] (1, 2) circle (4pt); 
    \end{scope}

    \begin{scope}[xshift = 120]
        \draw[dashed] (0.5, 0) -- (0, 1);
        \edge{0, 1}{0.5, 2}
        \edge{0.5, 2}{0, 3}
        \edge{0, 3}{0, 4}

        \edge{0.5, 0}{1, 1}
        \edge{1, 1}{1.5, 2}
        \edge{1.5, 2}{1, 3}
        \edge{1, 3}{1, 4}

        \edge{1, 0}{1.5, 1}
        \edge{1.5, 1}{0.5, 2}
        \edge{0.5, 2}{1, 3}

        \edge{2, 0}{1.5, 1}
        \edge{1.5, 1}{1.5, 2}
        \edge{1.5, 2}{2, 3}
        \edge{2, 3}{2, 4}

        \draw[dashed] (1, 3) to[bend right = 50] (1, 4);
        
        \vertex[2]{0.5, 0}
        \vertex[2]{0.5, 2}
        \vertex[2]{1, 3}
        \vertex[2]{1, 4}
        \vertex[2]{1.5, 1}
        \vertex[2]{1.5, 2}
    \end{scope}

    \begin{scope}[xshift = 240, yscale = -1, yshift = -114]
        \draw[dashed] (0.5, 0) -- (0, 1);
        \edge{0, 1}{0.5, 2}
        \edge{0.5, 2}{0, 3}
        \edge{0, 3}{0, 4}

        \edge{0.5, 0}{1, 1}
        \edge{1, 1}{1.5, 2}
        \edge{1.5, 2}{1, 3}
        \edge{1, 3}{1, 4}

        \edge{1, 0}{1.5, 1}
        \edge{1.5, 1}{0.5, 2}
        \edge{0.5, 2}{1, 3}

        \edge{2, 0}{1.5, 1}
        \edge{1.5, 1}{1.5, 2}
        \edge{1.5, 2}{2, 3}
        \edge{2, 3}{2, 4}

        \draw[dashed] (1, 3) to[bend right = 50] (1, 4);
        
        \vertex[2]{0.5, 0}
        \vertex[2]{0.5, 2}
        \vertex[2]{1, 3}
        \vertex[2]{1, 4}
        \vertex[2]{1.5, 1}
        \vertex[2]{1.5, 2}
    \end{scope}

    \begin{scope}[yshift = -170]
        \edge{-2, 0}{-2, 4}
        \vertex{-2, 1}
        \vertex{-2, 2}
        \vertex{-2, 3}
        \node[anchor = east] at (-2.2, 0.5) {$d$};
        \node[anchor = east] at (-2.2, 1.5) {$c$};
        \node[anchor = east] at (-2.2, 2.5) {$b$};
        \node[anchor = east] at (-2.2, 3.5) {$a$};
        \draw[->, black, thick] (-0.4, 2) -- (-1.6, 2);
    
        \edge{0, 0}{0.5, 1}
        \edge{0.5, 1}{0, 2}
        \edge{0, 2}{0.5, 3}
        \edge{0.5, 3}{0, 4}

        \edge{1, 0}{0.5, 1}
        \edge{0.5, 1}{1, 2}
        \edge{1, 2}{0.5, 3}
        \edge{0.5, 3}{1, 4}

        \vertex[2]{0.5, 1}
        \vertex[2]{0.5, 3}

        \edge{2, 0}{1.5, 1}
        \edge{1.5, 1}{1, 2}
        \edge{1, 2}{1.5, 3}
        \edge{1.5, 3}{2, 4}
        \foreach \y in {0, 1, 2, 3}{
            \edge{2, \y}{2, \y + 1}
        }

        \vertex[2]{2, 0}
        \vertex[2]{2, 4}   

        \filldraw[RedOrange, thick] (1, 2) circle (4pt); 
    \end{scope}

    \begin{scope}[yshift = -170, xshift = 120]
        \draw[dashed] (0.5, 0) -- (0, 1);
        \edge{0, 1}{0.5, 2}
        \edge{0.5, 2}{0, 3}
        \draw[dashed] (0, 3) -- (0.5, 4);

        \edge{0.5, 0}{1, 1}
        \edge{1, 1}{1.5, 2}
        \edge{1.5, 2}{1, 3}
        \edge{1, 3}{0.5, 4}

        \edge{1, 0}{1.5, 1}
        \edge{1.5, 1}{0.5, 2}
        \edge{0.5, 2}{1.5, 3}
        \edge{1.5, 3}{1, 4}

        \edge{2, 0}{1.5, 1}
        \edge{1.5, 1}{1.5, 2}
        \edge{1.5, 2}{1.5, 3}
        \edge{1.5, 3}{2, 4}    
        
        \vertex[2]{0.5, 0}
        \vertex[2]{0.5, 2}
        \vertex[2]{0.5, 4}
        \vertex[2]{1.5, 1}
        \vertex[2]{1.5, 2}
        \vertex[2]{1.5, 3}
    \end{scope}

    \begin{scope}[yshift = -170, xshift = 240]
        \edge{0, 0}{0, 1}
        \edge{0, 1}{0.5, 2}
        \edge{0.5, 2}{0, 3}
        \edge{0, 3}{0, 4}

        \edge{1, 0}{1, 1}
        \edge{1, 1}{0.5, 2}
        \edge{0.5, 2}{1, 3}
        \edge{1, 3}{1, 4}

        \draw[dashed] (1, 0) to[bend right = 50] (1, 1);
        \edge{1, 1}{1.5, 2}
        \edge{1.5, 2}{1, 3}
        \draw[dashed] (1, 3) to[bend right = 50] (1, 4);

        \edge{2, 0}{2, 1}
        \edge{2, 1}{1.5, 2}
        \edge{1.5, 2}{2, 3}
        \edge{2, 3}{2, 4}

        \vertex[2]{0.5, 2}
        \vertex[2]{1, 0}
        \vertex[2]{1, 1}
        \vertex[2]{1, 3}
        \vertex[2]{1, 4}
        \vertex[2]{1.5, 2}
    \end{scope}
\end{tikzpicture}
    \caption{The towers of \cref{fig:discont_example} with the edge $e$ and all its preimages contracted.}
    \label{fig:discont_example_contracted}
\end{figure}
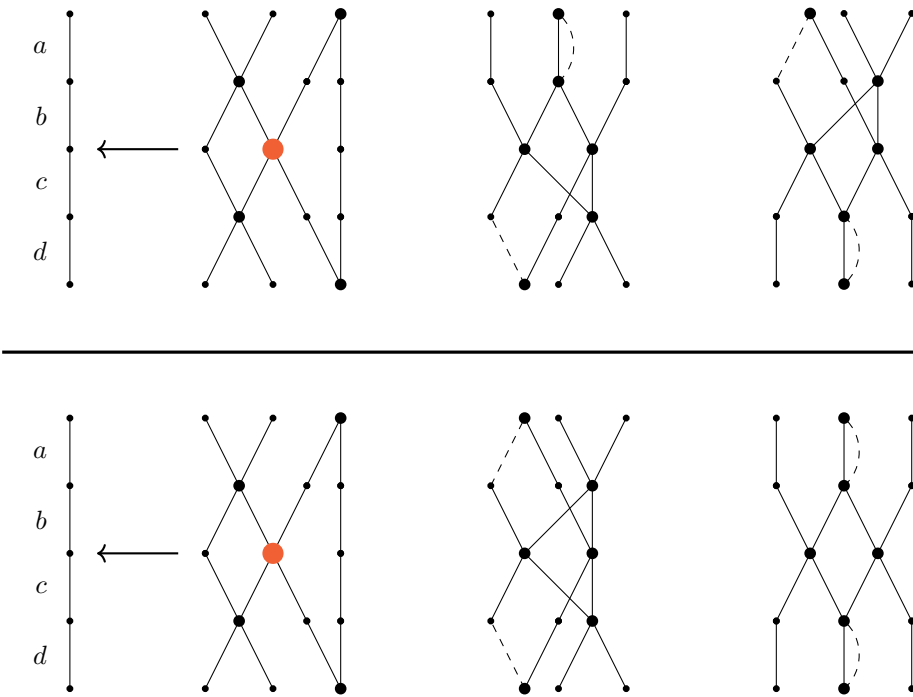

From this example we conclude that one cannot give a new definition of the tetragonal construction that matches the current one for good towers, which gives rise to a triality, and which is continuous with respect to edge contractions. 
    
\FloatBarrier

\addcontentsline{toc}{section}{Bibliography}
\bibliographystyle{alpha}
\bibliography{bibliography}

\end{document}